\documentclass[10pt,a4paper]{article}
\usepackage[top=1.25in, bottom=1.0in, left=1.0in, right=1.0in]{geometry}

\usepackage[T1]{fontenc}
\usepackage[utf8]{inputenc}
\usepackage{authblk}
\usepackage{geometry}
\usepackage{amssymb}
\usepackage{amsmath}
\numberwithin{equation}{section} %修改公式编号方案
\usepackage{amsthm}
\usepackage{diagbox}
\usepackage{empheq}
\usepackage{mdframed}
\usepackage{multirow}
\usepackage{mathrsfs}
\usepackage{booktabs}
\usepackage{lipsum}
\usepackage{graphicx}
\usepackage{color}
\usepackage{arydshln}
\usepackage{psfrag}
\usepackage{pgfplots}
\usepackage{bm}
\usepackage{float}
\usepackage{stmaryrd} 
\usepackage{subfigure}
\usepackage{url}
\usepackage{appendix}

\definecolor{mm}{RGB}{131,49,154}
\definecolor{gg}{RGB}{0,150,0}
\definecolor{bb}{RGB}{33,77,169}
\usepackage{hyperref}
\hypersetup{
	colorlinks=true,
	linkcolor=blue,
	citecolor=blue,
}

\newtheorem{thm}{Theorem}[section]

\newtheorem{den}{Definition}[section]
\newtheorem{rmk}{Remark}[section]
\newtheorem{lem}{Lemma}[section]

\newtheorem{alg}{Algorithm}[section]

\newenvironment{keywords}
{\par\noindent\textbf{Keywords:}}

\graphicspath{{fig/}}
\definecolor{ocre}{RGB}{243,102,25}
\definecolor{mygray}{RGB}{243,243,244}
\definecolor{deepGreen}{RGB}{26,111,0}
\definecolor{shallowGreen}{RGB}{235,255,255}
\definecolor{deepBlue}{RGB}{61,124,222}
\definecolor{shallowBlue}{RGB}{235,249,255}
\newtheoremstyle{mytheoremstyle}{3pt}{3pt}{\normalfont}{0cm}{\rmfamily\bfseries}{}{1em}{{\color{black}\thmname{#1}~\thmnumber{#2}}\thmnote{\,--\,#3}}
\newtheoremstyle{myproblemstyle}{3pt}{3pt}{\normalfont}{0cm}{\rmfamily\bfseries}{}{1em}{{\color{black}\thmname{#1}~\thmnumber{#2}}\thmnote{\,--\,#3}}
\theoremstyle{mytheoremstyle}
\newmdtheoremenv[linewidth=1pt,backgroundcolor=shallowGreen,linecolor=deepGreen,leftmargin=0pt,innerleftmargin=20pt,innerrightmargin=20pt,]{theorem}{Theorem}[section]
\theoremstyle{mytheoremstyle}
\newmdtheoremenv[linewidth=1pt,backgroundcolor=shallowBlue,linecolor=deepBlue,leftmargin=0pt,innerleftmargin=20pt,innerrightmargin=20pt,]{definition}{Definition}[section]
\theoremstyle{myproblemstyle}
\newmdtheoremenv[linecolor=black,leftmargin=0pt,innerleftmargin=10pt,innerrightmargin=10pt,]{problem}{Problem}[section]
\usepgfplotslibrary{colorbrewer}
\pgfplotsset{width=8cm,compat=1.9}
\title{\LARGE A novel high-order linearly implicit and energy-stable additive Runge-Kutta methods for gradient flow models.}
\author[a]{Xuelong Gu}
\author[a]{Wenjun Cai}
\author[a]{Yushun Wang \thanks{Corresponding author: wangyushun@njnu.edu.cn}}
\affil[a]{Ministry of Education Key Laboratory for NSLSCS \\ Jiangsu Collaborative Innovation Center of Biomedical Functional Materials \\ School of Mathematical Sciences \\ 
Nanjing Normal University, Nanjing, 210023, China.  \\ \vspace{-2cm} }
\date{}

\begin{document}
\maketitle
{\noindent}	 \rule[-10pt]{15.5cm}{0.1em}
\begin{abstract}
	This paper introduces a novel paradigm for constructing linearly implicit and high-order unconditionally energy-stable schemes for general gradient flows, utilizing the scalar auxiliary variable (SAV) approach and the additive Runge-Kutta (ARK) methods. We provide a rigorous proof of energy stability, unique solvability, and convergence. The proposed schemes generalizes some recently developed high-order, energy-stable schemes and address their shortcomings.
	On the one other hand, the proposed schemes can incorporate existing SAV-RK type methods after judiciously selecting the Butcher tables of ARK methods \cite{sav_li,sav_nlsw}. The order of a SAV-RKPC method can thus be confirmed theoretically by the order conditions of the corresponding ARK method. Several new schemes are constructed based on our framework, which perform to be more stable than existing SAV-RK type methods. On the other hand, the proposed schemes do not limit to a specific form of the nonlinear part of the free energy and can achieve high order with fewer intermediate stages compared to the convex splitting ARK methods \cite{csrk}.
	 Numerical experiments demonstrate stability and efficiency of proposed schemes.
\end{abstract}
\begin{small}
\end{small}
{\noindent}	 \rule[-10pt]{15.5cm}{0.1em}
\section{Introduction}

Phase field models are versatile mathematical equations widely used in physics, material science, and mathematics to simulate various physical phenomena, including the diffusion of two-phase interfaces, phase transitions in materials, and mechanical properties \cite{001,002,003,004}. These models are useful for describing different phases of material and the phase transitions and microstructural changes that occur in non-equilibrium states. 

%In recent years, phase field models have been widely used in alloy casting, new material preparation, image processing, and finance.

The phase field model is usually represented as a gradient flow of a free energy functional $\mathcal{F}(u)$ as follows:
\begin{equation}\label{eq:intr_gradient_flow}
\begin{aligned}
    \frac{\partial u}{\partial t}= \mathcal{G} \frac{\delta \mathcal{F}}{\delta u}, \quad (\mathbf{x}, t) \in \Omega \times (0, T], \\
\end{aligned}
\end{equation}
with the initial condition $u(\mathbf{x}, 0) = u_0(\mathbf{x})$, where $u$ is a state variable, $\Omega \subset \mathbb{R}^n$ represents the computational domain, $\frac{\delta \mathcal{F}}{\delta u}$ denotes the variational derivative of $\mathcal{F}$ to $u$, and $\mathcal{G}$ is a non-positive mobility operator. Classical phase field models include the Allen-Cahn (AC) equation \cite{intr_ac}, the Cahn-Hilliard (CH) equation \cite{intr_ch}, the molecular beam epitaxy (MBE) equation \cite{intr_mbe}, etc \cite{intr_other1,intr_crystal,intr_other2}. A significant aspect of \eqref{eq:intr_gradient_flow} is that the system preserves the following energy dissipation law when appropriate boundary conditions are imposed on $u$.
\begin{equation}\label{eq:intr_edl}
    \frac{d \mathcal{F}}{dt} = (\frac{\delta \mathcal{F}}{\delta u}, \frac{\partial u}{\partial t}) = (\frac{\delta \mathcal{F}}{\delta u}, \mathcal{G} \frac{\delta \mathcal{F}}{\delta u}) \leq 0.
\end{equation}

Due to the nonlinearity of \eqref{eq:intr_gradient_flow}, its analytical solution is typically intractable. Therefore, developing efficient and stable numerical schemes is imperative. One approach is constructing schemes that inherit a discrete counterpart of \eqref{eq:intr_edl}, known as energy-stable methods \cite{grad_stable_ch}. As demonstrated in \cite{dvd}, energy-stable methods can prevent numerical oscillations and unphysical solutions, thus have been the focus of extensive researchers over the past few decades. Classical energy-stable methods include convex splitting (CS) methods \cite{convex_splitting1,grad_stable_ch,liao_bdf,csrk} and discrete variational derivative (DVD) methods \cite{dvd2,dvd_high,dvd1,qiao_mixed_fe} and so on. CS and DVD methods are fully implicit, thus requiring solving a nonlinear system at each time step. To improve computational efficiency, researchers have suggested linearly implicit or explicit energy-stable schemes, such as stabilized semi-implicit methods \cite{cn_ab,tang_imex}, exponential time difference methods \cite{du_2019,du_2021,mbe_etd3,ju_mbe}, and the leapfrog methods \cite{mbe_leapfrog,hou_leapfrog}.
The numerical methods discussed above are exclusive to particular gradient flow models and can not be effortlessly adapted to others.  This status quo did not change until the energy quadratization (EQ) methods \cite{ieq1,ieq2,ieq3,ieq4} were proposed. EQ methods provide an elegant platform for constructing linearly implicit schemes, but they involve solving linear systems with variable coefficients at each time step. In \cite{sav_shen}, Shen et al. proposed scalar auxiliary variable (SAV) methods.  Besides their unconditional stability, SAV methods require only the solution of a linear system with constant coefficients in each step. Furthermore, SAV approaches provide a universal framework for developing linearly implicit energy-stable schemes that can be extended to a variety of complex models \cite{sav_ns_err,sav_ns_err2,sav_ch3,sav_chhs_err}. Due to these advantages, SAV methods have received attention and are promoted in \cite{lag,GSAV1,svm,GSAV2}.
%The popularity of CS methods can be attributed to two main advantages: (i) A CS method is unconditionally energy-stable without severe limitations on the time step. (ii) The resulting nonlinear system can be readily solved (e.g., Newton iteration will converge, regardless of the initial guess). However, the construction of CS methods depends on the specific form of the equation, and higher-order CS methods are complicated. The DVD methods can be used to construct energy-stable methods for gradient flows with complicated free energy functional. However, their solvability can only be guaranteed when the time step is sufficiently small. CS and DVD methods are fully implicit, thus requiring solving a nonlinear system at each time step. To improve computational efficiency, researchers have suggested linearly implicit or explicit energy-stable schemes, such as stabilized semi-implicit methods \cite{cn_ab,tang_imex}, exponential time difference methods \cite{du_2019,du_2021,mbe_etd3,ju_mbe}, and the leapfrog methods \cite{mbe_leapfrog,hou_leapfrog}.

However, the above methods are limited to second-order accuracy, which may not accommodate high precision requirements. The nonlinearity of phase field models makes it difficult to develop high-order energy-stable schemes. In \cite{csrk}, the authors present high-order energy-stable schemes by combining additive Runge-Kutta (ARK) methods with CS techniques (CS-ARK). To guarantee energy stability, these approaches impose stringent criteria on the coefficients of the ARK methods, necessitating a large number of intermediate stages even for a second-order scheme. Thus, the currently identified energy-stable CS-ARK methods are limited to third-order. In \cite{ac_hbvm,dvd_high}, energy-stable schemes based on the Hamiltonian boundary value or discrete gradient methods are presented. These schemes are fully implicit and thus computationally expensive. Akrivis et al. introduced in \cite{sav_rk_extra} novel linearly implicit schemes based on a combination of SAV and RK (SAV-RK) approaches. For explicit discretization of nonlinear terms, they incorporated extrapolation techniques to predict solutions at specified time levels. The resulting methods are referred to as SAV-RKEX. However, excessive interpolation points lead to highly oscillatory interpolation polynomials, resulting in inaccurate predictions. Li et al. developed SAV-RKPC methods in \cite{gong_nls} to obtain a more accurate prediction of numerical solutions at intermediate stages, significantly improving the stability and accuracy of SAV-RKEX methods. Nevertheless, such a technique increases the computational costs, and there is no theoretical guarantee of the necessary number of iterations to achieve adequate accuracy.

In this paper, we propose a novel paradigm for constructing linearly implicit and high-order unconditionally energy-stable schemes, combining the SAV approach with the ARK methods. The proposed methods overcome the limitations of both CS-ARK and SAV-RK methods and can be applied to gradient flow systems with general nonlinear functionals. On the one hand, to guarantee energy stability, the proposed methods require only the algebraic stability of the implicit part of ARK methods. This enables the methods to achieve high accuracy and energy stability with fewer intermediate stages. On the other hand, our approach can be regarded as a novel prediction correction technique that avoids the imprecision of extrapolation techniques used in the SAV-RKEX method and does not require iterative procedures for prediction in SAV-RKPC. Thus, the proposed approach guarantees both efficiency and stability. Additionally, our framework can accommodate all SAV-RK type integrators with some appropriate modifications, enabling us to theoretically analyze the consistency of SAV-RKPC(or EQ) methods proposed in \cite{gong_nls,ieq_gong} by exploiting the order conditions of ARK methods.

The overall structure of the remaining contexts is summarized below. In Section \ref{sec2}, we briefly overview the ARK and SAV methods. In Section \ref{sec3}, we reformulate the gradient flow model into an equivalent one and propose our new algorithms. Then, we prove the unconditional energy stability and solvability of the proposed methods. Moreover, we demonstrate the order condition of SAV-RKPC methods by regarding it as an ARK method. The numerical examples and comparisons are made in Section \ref{sec5}. Finally, we conclude the whole work in Section \ref{sec6}.

\section{Overview of ARK methods and SAV reformulation of gradient flows}\label{sec2}
In this section, we briefly overview the additive Runge-Kutta (ARK) methods. Some basic notations and concepts are also presented. By incorporating a scalar auxiliary variable, the original gradient flow model is transformed into an equivalent one (known as the SAV reformulation). The reformulated system preserves the quadratic energy and provides an elegant platform for developing high-order and linearly implicit unconditionally energy-stable numerical methods.

\subsection{ARK methods}
We provide an overview of ARK methods, which are commonly used to solve the initial value problem for the following additive partitioned system:
\begin{equation}\label{model_ode}
    u_t(\mathbf{x}, t) = f(u) + g(u), \quad u(\mathbf{x}, 0) = u_0(\mathbf{x}).
\end{equation}
Here, the right-hand side of \eqref{model_ode} is subdivided with respect to stiffness, nonlinearity, dynamical behavior, etc.  Before we proceed, it is helpful to introduce the Butcher notations for two $s$-stage RK methods.
\begin{small}
\begin{equation} \label{butcher}
\begin{array}{c| c}
  c & A \\
  \hline
  & b^{\mathrm{T}}
\end{array}
=
\begin{array}{c| c c c c }
	c_0 & a_{00} & \cdots & a_{0s-1} \\
	c_1 & a_{10} & \cdots & a_{1s-1} \\
	\vdots  & \vdots & \cdots & \vdots \\
	c_{s-1} &  a_{s-1 0}  & \cdots & a_{s-1 s-1} \\
	\hline
		& b_0 & \cdots & b_{s-1}
\end{array}
, \quad
\begin{array}{c| c}
	\widehat{c} & \widehat{A} \\
	\hline
			& \widehat{b}^{\mathrm{T}}
\end{array}
= 
\begin{array}{c| c c c }
	\widehat{c}_0 & \widehat{a}_{00}  & \cdots & \widehat{a}_{0s-1} \\
	\widehat{c}_1 & \widehat{a}_{10}  & \cdots & \widehat{a}_{1s-1} \\
	\vdots & \vdots & \cdots & \vdots \\
	\widehat{c}_{s-1} &  \widehat{a}_{s-1 0}  & \cdots & \widehat{a}_{s-1 s-1} \\
	\hline
		& \widehat{b}_0  & \cdots & \widehat{b}_{s-1}
\end{array}
,
\end{equation}
\end{small}
where $A \in \mathbb{R}^{s \times s}, b \in \mathbb{R}^s$, and $c = A \mathbf{1}$ with $\mathbf{1} = (1, 1, \cdots, 1)^\mathrm{T} \in \mathbb{R}^s$. $\widehat{A}$, $\widehat{c}$ are defined in the similar manner.
\begin{den}[Explicit RK (ERK) methods]
    A RK method is explicit if $a_{ij} = 0$ for $j \geq i-1$.
\end{den}
\begin{den}[Diagonally implicit RK (DIRK) methods]
    A RK method is diagonally implicit if $a_{ij} = 0$ for $j \geq i$ and there exists $0 \leq i \leq s-1, a_{ii} \neq 0$.
\end{den}
\begin{den}[algebraically stable RK method \cite{burrage_efficiently_1982, burrage_stability_1979,ieq_gong}]
    Let us consider a symmetric matrix with entries $M_{ij} = b_i a_{ij} + b_j a_{ji} - b_i b_j$. A RK method is algebraically stable if its coefficients satisfy the following stability criteria.
\begin{itemize}
    \item $b_i \geq 0, \ \forall i = 1, 2, \cdots, s$,
    \item $M$ is positive semi-definite.
\end{itemize}
\end{den}
We partition the time interval uniformly with a step size of $\tau$ and denote the time grid points as $t_n = n \tau$. Let $N_t = \left[\frac{T}{\tau}\right]$. Assuming that $u^n$ has been solved in advance. The ARK methods then update $u^{n+1}$ through two steps. 

First, the intermediate stages $u_{ni} \ (i = 0, 1, \cdots, s-1)$ are computed from
\begin{equation*}
        u_{ni} = u^n + \tau \sum\limits_{j=0}^{s-1} a_{ij} f(u_{nj}) + \tau \sum\limits_{j=0}^{s-1} \widehat{a}_{ij} g(u_{nj}),
\end{equation*}
Then, we update the solution by
\begin{equation*}
    u^{n+1} = u^n + \tau \sum\limits_{i=0}^{s-1} b_i f(u_{ni}) + \tau\sum\limits_{i=0}^{s-1} \widehat{b}_i g(u_{ni}).
\end{equation*}

It is worth mentioning that the above ARK methods have been employed to develop energy-stable schemes for phase field models in \cite{csrk} and maximum bound principle methods for the AC equations in \cite{imex_fac}.

\begin{rmk}\label{rmk_prk}
  We emphasize that each ARK method can be considered as a partitioned Runge-Kutta (PRK) method \cite{hairer_book}. Specifically, let us introduce an equivalent reformulation of \eqref{model_ode} as follows:
\begin{equation}\label{imexrk_equivalent_system}
\left\lbrace
\begin{aligned}
    &\dot{u}_f(\mathbf{x}, t) = f(u), \ \dot{u}_g(\mathbf{x}, t) = g(u), \\ 
    & u(\mathbf{x}, t) = u_f(\mathbf{x}, t) + u_g(\mathbf{x}, t),
\end{aligned}
\right.
\end{equation}
 It is straightforward to see that \eqref{imexrk_equivalent_system} is equivalent to \eqref{model_ode} if the consistent initial condition $u_f(\mathbf{x}, 0) + u_g(\mathbf{x}, 0) = u^0(\mathbf{x})$ is imposed. By employing a PRK method to \eqref{imexrk_equivalent_system} and eliminating the intermediate variables $u_f$, $u_g$, we readily obtain the ARK method as mentioned above.
\end{rmk}

By Remark \ref{rmk_prk}, we can readily infer that an ARK method has an order of $p$ if the corresponding PRK method has an order of $p$, as a ARK method is essentially a PRK method applied to the extended systems \eqref{imexrk_equivalent_system}. Adrian et al. conducted an extensive study on generalized ARK methods in \cite{ark_general} and provided a comprehensive list of their order conditions. Table \ref{order_condition} summarizes the order conditions of ARK methods up to the third-order for convenience.
  \begin{table}[H]\label{order_condition}
  \begin{center}
  \caption{Order conditions of ARK methods up to third order accuracy, where $b \cdot c = \sum\limits_{i=0}^{s-1} b_i c_i$, 
  	$b \odot c = (b_0c_0, \cdots b_{s-1}c_{s-1})^\mathrm{T}$, and $c^2 = c \odot c$.}
  \begin{tabular}{ c c c c }
  \toprule
  Order & \multicolumn{2}{c}{Stand-alone conditions} & Coupling  conditions \\
   \hline
  1 &  $b \cdot \bm{1} = 1$ & $\widehat{b} \cdot \bm{1} = 1$ & $\star$ \\
  \hdashline
  2 &  $b \cdot c = 1/2$ & $\widehat{b} \cdot \widehat{c} = 1/2$ & $b \cdot \widehat{c} = \widehat{b} \cdot c = 1/2$ \\
  \hdashline
  \multirow{5}{*}{3} & $b \cdot c^2 = 1/3$ & $\widehat{b} \cdot \widehat{c}^2 = 1/3$ & $b \cdot (c \odot \widehat{c}) = b \cdot (\widehat{c} \odot c) = 1/3$ \\
    & $b \cdot Ac = 1/6$ & $\widehat{b} \cdot \widehat{A}\widehat{c} = 1/6$ & $\widehat{b} \cdot (c \odot \widehat{c}) = \widehat{b} \cdot (\widehat{c} \odot c)  = 1/3$ \\
    &  &  & $b \cdot \widehat{c}^2 = \widehat{b} \cdot c^2 = 1/3$ \\
    &  &  & $b \cdot A\widehat{c} = b \cdot \widehat{A} c = b \cdot \widehat{A}\widehat{c} =   1/6$ \\
    & &  & $\widehat{b} \cdot A\widehat{c} = \widehat{b} \cdot \widehat{A} c = \widehat{b} \cdot A c =   1/6$ \\
  \bottomrule
\end{tabular}
\end{center}		
\end{table}

\subsection{Gradient flow systems and their SAV reformulation}
A gradient flow model can be expressed generally as
\begin{equation}\label{general_gradient_flow}
    u_t(\mathbf{x}, t) = \mathcal{G} \frac{\delta \mathcal{F}}{\delta u}, \ \mathbf{x} \in \Omega,
\end{equation}
where $u$ is a state variable, $\mathcal{G} \in \mathbb{R}^{d \times d}$ is a negative semi-definite mobility operator, and $\frac{\delta \mathcal{F}}{\delta u}$ is the variational derivative of the free energy functional $\mathcal{F}$  to $u$. The triple $(u, \mathcal{G}, \mathcal{F})$ uniformly specifies a gradient flow system. When appropriate boundary conditions are imposed on $u$, system \eqref{general_gradient_flow} dissipates the free energy as follows:
\begin{equation}
    \frac{d \mathcal{F}}{dt} = \left( \frac{\delta \mathcal{F}}{\delta u}, \frac{\partial u}{\partial t} \right) = \left( \frac{\delta \mathcal{F}}{\delta u}, \mathcal{G} \frac{\delta \mathcal{F}}{\delta u} \right) \leq 0,
\end{equation}
where $(u, v) = \int_\Omega u v d\mathbf{x}, \ \forall u, v \in L^2(\Omega)$ is the inner product. Moreover, we denote by $\|u\| = \sqrt{(u, u)}$ the corresponding norm. 

For illustration, let us assume a free energy functional of the form:
\begin{equation}\label{original_energy}
    \mathcal{F}(u, \nabla u) = \frac{1}{2}(u, \mathcal{L}u) + (F(u, \nabla u), 1),
\end{equation}
where $\mathcal{L}$ is a linear, self-adjoint, and positive definite operator, $F$ represents a bulk energy bounded below. The SAV approach introduces a new scalar variable such that
\begin{equation}\label{q}
    q(t) = \sqrt{ (F(u, \nabla u), 1) + C },
\end{equation}
where $C$ is a sufficiently large positive constant to guarantee that the square root in \eqref{q} makes sense. The energy functional \eqref{original_energy} can be rewritten into a quadratic form as
\begin{equation}\label{eq-quadratic_energy}
    \mathcal{F}(u, q) = \frac{1}{2}(u, \mathcal{L}u) + q^2 - C.
\end{equation}
Let $W(u) = \sqrt{ (F(u, \nabla u), 1) + C}$ for simplicity. The model \eqref{general_gradient_flow} is reformulated into an equivalent system using the SAV approach \cite{sav_rk_extra,sav_shen,sav_shen_siam}, as shown below:
\begin{equation}\label{sav_reformulation}
\left\lbrace
\begin{aligned}
   u_t &= \mathcal{G} (\mathcal{L}u + 2q \frac{\delta W}{\delta u} - 2q \nabla \cdot \frac{\delta W}{\delta \nabla u} ), \\
    q_t &= (\frac{\delta W}{\delta u}, u_t ) + ( \frac{\delta W}{\delta \nabla u}, \nabla u_t ),
\end{aligned}
\right.
\end{equation}
equipped with the consistent initial conditions
\begin{equation}\label{sav_initial_condition}
    u(\mathbf{x}, 0) = u_0(\mathbf{x}), \quad q(0) = \sqrt{(F(u_0, \nabla u_0), 1) + C}.
\end{equation}
Taking the inner products on both sides of the first and second equations of \eqref{sav_reformulation} by $\mathcal{L}u + 2q \frac{\delta W}{\delta u} - 2q \nabla \cdot \frac{\delta W}{\delta \nabla u}$ and $2q$, respectively, and then combining the resulting equations, it is readily to confirm that system \eqref{sav_reformulation} admits the following energy dissipation law.
\begin{equation}
    \frac{d}{dt} \mathcal{F}(u, q) = \big( \mathcal{L}u + 2q \frac{\delta W}{\delta u} - 2q \nabla \cdot \frac{\delta W}{\delta \nabla u}, \mathcal{G} (\mathcal{L}u + 2q \frac{\delta W}{\delta u} - 2q \nabla \cdot \frac{\delta W}{\delta \nabla u} )\big) \leq 0.
\end{equation}
\section{High-order linearly implicit and energy-stable schemes}\label{sec3}
\subsection{Construction of time integrators} \label{sec3-1}
Let us further reformulate \eqref{sav_reformulation} as follows:
\begin{equation} \label{sav_tark_reformulation}
\left\lbrace
\begin{aligned}
    v_t &= \mathcal{G} \big( \mathcal{L}v + 2q \frac{\delta W}{\delta u}[v] - 2q \nabla \cdot \frac{\delta W}{\delta \nabla u}[v] \big), \\
    u_t &= \mathcal{G} \big( \mathcal{L}u + 2 q \frac{\delta W}{\delta u}[v] - 2q \nabla \cdot \frac{\delta W}{\delta \nabla u}[v]\big), \\
    q_t &= \big( \frac{\delta W}{\delta u}[v], u_t \big) + \big( \frac{\delta W}{\delta \nabla u}[v], \nabla u_t \big),
\end{aligned}
\right.
\end{equation} 
equipped with the initial conditions
\begin{equation} \label{sav_tark_initial_condition}
    u(\mathbf{x}, 0) = v(\mathbf{x}, 0) = u_0(\mathbf{x}), \quad q(0) = \sqrt{(f(u_0(\mathbf{x}), \nabla u_0(\mathbf{x})), 1) + C}.
\end{equation}
We first demonstrate the equivalence between the reformulated system \eqref{sav_tark_reformulation}, \eqref{sav_tark_initial_condition} and the original system \eqref{general_gradient_flow}.
\begin{lem}\label{lem_equivalence}
    Suppose that $\mathcal{L}$ is a linear, self-adjoint, and positive definite operator. The reformulation \eqref{sav_tark_reformulation} and the initial condition \eqref{sav_tark_initial_condition} are equivalent to \eqref{general_gradient_flow}.
\end{lem}
\begin{proof}
    According to the definition of $q$ \eqref{q} and introducing $v(t) = u(t)$, it is evident that the original system \eqref{general_gradient_flow} implies \eqref{sav_tark_initial_condition}. We will now demonstrate that the combination of \eqref{sav_tark_initial_condition} and \eqref{sav_tark_initial_condition} leads to \eqref{general_gradient_flow}. Subtracting the second equation from the first equation of \eqref{sav_tark_reformulation} yields
    \begin{equation*}
        u_t - v_t = \mathcal{G} \big( \mathcal{L} u - \mathcal{L} v \big).
    \end{equation*}
    Taking the discrete inner product with $\mathcal{L} u - \mathcal{L} v$ on both sides of the above equation produces
    \begin{equation*}
        \frac{1}{2}\frac{d}{dt} \big(\mathcal{L}(u - v), u - v \big) = \big(\mathcal{G} \mathcal{L}(u - v), \mathcal{L}(u - v)\big) \leq 0.
    \end{equation*}
   Due to the positive-definite of $\mathcal{L}$ and \eqref{sav_tark_initial_condition}, we conclude that
    \begin{equation}\label{solution_eq}
       u(t) = v(t), \quad \forall \ 0 \leq t \leq T.
    \end{equation}
    Inserting \eqref{solution_eq} into the third equation of \eqref{sav_tark_reformulation}, we obtain
    \begin{equation}\label{qt}
\begin{aligned}
    q_t = \big( \frac{\delta W}{\delta u}[v]  , v \big) + \big( \frac{\delta W}{\delta \nabla u}[v], \nabla v_t \big) = \frac{d W[v]}{dt}.
\end{aligned}
\end{equation}
Combining \eqref{sav_tark_initial_condition}, \eqref{solution_eq}, and \eqref{qt} results in
\begin{equation*}
    q = W[v] = W[u],
\end{equation*}
Finally, it holds from the definition of $W$ that
\begin{equation*}
    2q \frac{\delta W}{\delta u} = \frac{\delta F}{\delta u}, \quad 2q \nabla \cdot \frac{\delta W}{\delta \nabla u} = \nabla \cdot \frac{\delta F}{\delta \nabla u}.
\end{equation*}
Substituting the above results into \eqref{sav_tark_reformulation} yields \eqref{general_gradient_flow}, which completes the proof.
\end{proof}
\begin{rmk}
    The positive-definite of $\mathcal{L}$ is reasonable for most phase field models. For the CH equation with Neumann or periodic boundary conditions, we have $\mathcal{L} = -\Delta$ and $\mathcal{G} = \Delta$. The mass conservation law guarantees the invertibility of $\mathcal{L}$. A similar argument applies to the MBE equation. For the AC equation, we have $\mathcal{L} = -\Delta$ and $\mathcal{G} = -I$. Although $\mathcal{L}$ is only positive semi-definite in this case, we can introduce a stabilized parameter $\kappa$ and equivalently recast the AC equation as
    \begin{equation*}
       u_t = - \big( (\kappa I - \Delta) u - (\kappa u + f(u)) \big) := - \big(\mathcal{L}_\kappa u + f_\kappa (u)\big).
    \end{equation*}
   Then, $\mathcal{L}_\kappa = \kappa I - \Delta$ is positive definite.
\end{rmk}
\begin{rmk}
    The extension of \eqref{general_gradient_flow} results in a more complex system \eqref{sav_tark_reformulation}. However, this reformulation provides an elegant platform for developing high-order, linearly implicit, and energy-stable schemes, as will be demonstrated in subsequent contexts. It should be noted that the equivalent reformulation of \eqref{general_gradient_flow} is not unique, and other similar reformulations can be employed to develop numerical schemes through the frameworks described in this paper. For simplicity, we only consider \eqref{sav_tark_reformulation} in this section.
\end{rmk}
\begin{rmk}
    System \eqref{sav_tark_reformulation} is an extension of the original SAV approach \eqref{sav_reformulation} proposed in \cite{sav_shen}. Some other SAV approaches have recently gained popularity, including the exponential SAV approach \cite{ESAV} and the generalized SAV approach \cite{GSAV1, GSAV2}. In \cite{ESAV_AC}, Ju et al. have also introduced a novel exponential SAV approach to preserve both MBP and EDL for the AC equations. These approaches can also be extended similarly to \eqref{sav_tark_reformulation} and discretized by the methods outlined in subsequent contexts to obtain high-order and energy-stable schemes. For simplicity, we will only use the original SAV approach for illustrations.
\end{rmk}
Assuming that $u^n$, $v^n$, and $q^n$ are already determined. The SAV-ARK methods are outlined below:
\begin{alg}[SAV-ARK]\label{ark_sav_alg}
    The intermediate variables $v_{ni}, u_{ni}$, and $q_{ni}$ are solved from
\begin{equation}\label{ark_sav_intermediate}
\left\lbrace
\begin{aligned}
    v_{ni} &= v^n + \tau \sum\limits_{j=0}^{s-1} (a_{ij} \dot{v}^{\mathcal{L}}_{nj} +\widehat{a}_{ij} \dot{v}^\mathcal{N}_{nj}) , \\
    u_{ni} &= u^n + \tau \sum\limits_{j=0}^{s-1} a_{ij} \dot{u}_{nj}, \ q_{ni} = q^n + \tau\sum\limits_{j=0}^{s-1} a_{ij} \dot{q}_{nj}, \\
    \dot{v}^\mathcal{L}_{ni} &= \mathcal{G} \mathcal{L} v_{ni}, \ \dot{v}^\mathcal{N}_{ni} = \mathcal{G} ( 2q_{ni} \frac{\delta W}{\delta u}[v_{ni}] - 2q_{ni} \nabla \cdot \frac{\delta W}{\delta \nabla u}[v_{ni}] ), \\
    \dot{u}_{ni} &= \mathcal{G} ( \mathcal{L} u_{ni} + 2q_{ni} \frac{\delta W}{\delta u}[v_{ni}] - 2q_{ni} \nabla \cdot \frac{\delta W}{\delta \nabla u}[v_{ni}] ), \\
    \dot{q}_{ni} &=  (\frac{\delta W}{\delta u}[v_{ni}], \dot{u}_{ni} ) + ( \frac{\delta W}{\delta \nabla u}[v_{ni}], \nabla \dot{u}_{ni} ).
\end{aligned}
\right.
\end{equation}
Then, the solution at $t_{n+1}$ is
\begin{equation}\label{ark_sav_next}
    v^{n+1} = v^n + \tau \sum\limits_{i=0}^{s-1} b_i(\dot{v}_{ni}^\mathcal{L} + \dot{v}_{ni}^\mathcal{N}), \
    u^{n+1} = u^n + \tau \sum\limits_{i=0}^{s-1} b_i \dot{u}_{ni}, \
    q^{n+1} = q^n + \tau \sum\limits_{i=0}^{s-1} b_i \dot{q}_{ni}.
\end{equation}
\end{alg}
	We note here that linearly implicit schemes can be obtained by carefully choosing the RK coefficients in Algorithm \ref{ark_sav_alg}. One effective method is discretizing $u$ and $q$ with DIRK methods and $v$ with ERK methods. These methods will be referred to as SAV-DIARK methods in the subsequent contexts.

It is important to emphasize that by introducing $\bm{z} = (v, u, q)^\mathrm{T}$, Algorithm \ref{ark_sav_alg} can be regraded as ARK methods as follows:
    \begin{equation*}
    \begin{aligned}
        \bm{z}_{ni} &= \bm{z}^n + \tau \sum\limits_{j=0}^{s-1} (a_{ij} \Phi(\bm{z}_{nj}) + \widehat{a}_{ij} \Psi(\bm{z}_{nj}) ), \\
        \bm{z}^{n+1} &= \bm{z}^n + \tau \sum\limits_{i=0}^{s-1} b_i ( \Phi(\bm{z}_{ni}) +  \Psi(\bm{z}_{ni}) ), \\
    \end{aligned}
    \end{equation*}
    where 
    \begin{equation*}
    \begin{aligned} 
        \Phi(\bm{z}) =
        \left(
        \begin{smallmatrix}
            \mathcal{G}\mathcal{L} u \\ 
            \mathcal{G}\big( \mathcal{L}u + 2q \frac{\delta W}{\delta u}[v] - 2q \nabla \cdot \frac{\delta W}{\delta \nabla u}[v]\big) \\ 
            \big( \frac{\delta W}{\delta u}[v], \dot{u} \big) + \big( \frac{\delta W}{\delta \nabla u}, \nabla \dot{u}\big)
        \end{smallmatrix} \right), \
         \Psi(\bm{z}) = 
        \left(
        \begin{smallmatrix}
            \mathcal{G} \big( 2q \frac{\delta W}{\delta u}[v] - 2q \nabla \cdot \frac{\delta W}{\delta \nabla u}[v] \big) \\
            0 \\ 
            0
        \end{smallmatrix} \right). \\
    \end{aligned}
    \end{equation*}
This allows us to easily derive the order conditions of the proposed schemes by the order conditions of ARK methods.

To further simplify and improve the stability of Algorithm \ref{ark_sav_alg}, we introduce the following modified SAV-ARK (SAV-MARK) scheme.
\begin{alg}[SAV-MARK]\label{mark_sav_alg}
    The intermediate variables $v_{ni}, u_{ni},  q_{ni}$ are solved from
\begin{equation}\label{mark_sav_intermediate}
\left\lbrace
\begin{aligned}
    v_{ni} &= u^n + \tau \sum\limits_{j=0}^{s-1} ( a_{ij} \dot{v}^\mathcal{L}_{nj} + \widehat{a}_{ij} \dot{v}^\mathcal{N}_{nj}) , \\
   u_{ni} &= u^n + \tau \sum\limits_{j=0}^{s-1} a_{ij} \dot{u}_{nj}, \ q_{ni} = q^n + \sum\limits_{j=0}^{s-1} a_{ij} \dot{q}_{nj}, \\
    \dot{v}^\mathcal{L}_{ni} &= \mathcal{G} \mathcal{L} v_{ni}, \ \dot{v}^\mathcal{N}_{ni} = \mathcal{G}( 2q_{ni} \frac{\delta W}{\delta u}[v_{ni}] - 2q_{ni} \nabla \cdot \frac{\delta W}{\delta \nabla u}[v_{ni}]), \\
    \dot{u}_{ni} &= \mathcal{G} ( \mathcal{L} u_{ni} + 2q_{ni} \frac{\delta W}{\delta u}[v_{ni}] - 2q_{ni} \nabla \cdot \frac{\delta W}{\delta \nabla u}[v_{ni}] ), \\
    \dot{q}_{ni} &=  \big(\frac{\delta W}{\delta u}[v_{ni}], \dot{u}_{ni} \big) + \big( \frac{\delta W}{\delta \nabla u}[v_{ni}], \nabla \dot{u}_{ni} \big).
\end{aligned}
\right.
\end{equation}
Then, the solution at $t_{n+1}$ is
\begin{equation}\label{mark_sav_next}
    u^{n+1} = u^n + \tau \sum\limits_{i=0}^{s-1} b_i \dot{u}_{ni}, \
    q^{n+1} = q^n + \tau \sum\limits_{i=0}^{s-1} b_i \dot{q}_{ni}.
\end{equation}
\end{alg}
In contrast to Algorithm \ref{ark_sav_alg}, Algorithm \ref{mark_sav_alg} does not require updating the variable $v$ at integer time steps. This modification not only reduces computational costs but also improves the stability of the scheme in practice. Additionally, thanks to \eqref{solution_eq}, this modification does not affect the accuracy of Algorithm \ref{ark_sav_alg}. 

\subsection{Energy stability and solvability}
\begin{thm}\label{thm-energy_stable}
    Suppose the RK methods employed on $u$ in Algorithms \ref{ark_sav_alg} and \ref{mark_sav_alg} are algebraically stable. Then, SAV-ARK and SAV-MARK methods are unconditionally energy-stable in the sense
    \begin{equation*}
        \mathcal{F}(u^{n+1}, q^{n+1}) \leq \mathcal{F}(u^n, q^n), \quad 0 \leq n \leq N_t - 1.
    \end{equation*}
\end{thm}
\begin{proof}
    By the definition of \eqref{ark_sav_next} and the self-adjointness of $\mathcal{L}$, we can derive
    \begin{equation*}
        \frac{1}{2} (u^{n+1}, \mathcal{L} u^{n+1}) - \frac{1}{2}(u^n, \mathcal{L} u^n) 
        = \tau  \sum\limits_{i=0}^{s-1} b_i (\dot{u}_{ni}, \mathcal{L} u^n) 
        + \frac{\tau^2}{2} \sum\limits_{i=0}^{s-1}\sum\limits_{j=0}^{s-1}  b_ib_j (\dot{u}_{ni}, \mathcal{L} \dot{u}_{nj}).
    \end{equation*}
    Substituting $u^n = u_{ni} - \tau \sum\limits_{j=0}^{s-1}a_{ij} \dot{u}_{nj}$ into the 
    above equation and observing that
    \begin{equation*}
        \sum\limits_{i = 0}^{s-1} \sum\limits_{j = 0}^{s-1} b_i a_{ij}(\dot{u}_{ni}, \mathcal{L} \dot{u}_{nj}) = \sum\limits_{i=0}^{s-1} \sum\limits_{j=0}^{s-1} b_ja_{ji}(\dot{u}_{ni}, \mathcal{L} \dot{u}_{nj}),
    \end{equation*}
    we obtain
    \begin{equation}\label{edl1}
    \begin{aligned}
        &\frac{1}{2} (u^{n+1}, \mathcal{L}u^{n+1}) - \frac{1}{2}(u^n, \mathcal{L} u^n) \\
        &\quad= \tau \sum\limits_{i=0}^{s-1} b_i (\dot{u}_{ni}, \mathcal{L}u_{ni}) - \frac{\tau^2}{2} \sum\limits_{i = 0}^{s-1} \sum\limits_{j=0}^{s-1} M_{ij} (\dot{u}_{ni}, \mathcal{L}\dot{u}_{nj}) \\ 
        &\quad \leq \tau \sum\limits_{i=0}^{s-1} b_i (\dot{u}_{ni}, \mathcal{L}u_{ni}).
    \end{aligned}
    \end{equation}
    The last inequality is a result of the positive definiteness of $M$ and $\mathcal{L}$. Using a similar procedure, we have
    \begin{equation}\label{edl2}
        (q^{n+1})^2 - (q^n)^2 \leq 2\tau \sum\limits_{i=0}^{s-1} b_i  q_{ni} \dot{q}_{ni}.
    \end{equation}
    Taking the discrete inner products of the sixth and last equations of \eqref{ark_sav_intermediate} with $\mathcal{L}u_{ni}$ and $2q_{ni}$, respectively, and adding the obtained results together yield
    \begin{equation}\label{edl3}
    \begin{aligned}
        (\dot{u}_{ni}, \mathcal{L}u_{ni}) + 2 q_{ni} \dot{q}_{ni} = \left(\mathcal{G}\mu_{ni}, \mu_{ni} \right) \leq 0,
    \end{aligned}
    \end{equation}
    where $\mu_{ni}  = \mathcal{L} u_{ni} + 2q_{ni} \frac{\delta W}{\delta u}[v_{ni}] - 2q_{ni} \nabla \cdot \frac{\delta W}{\delta \nabla u}[v_{ni}]$. The desired result is thus obtained by combining \eqref{edl1}--\eqref{edl3} with the condition $b_i \geq 0$.
\end{proof}
\begin{rmk}
    The proposed approach uses quadratic energy (as depicted in equation \eqref{eq-quadratic_energy}) instead of the original one. When higher-order time discretization is applied to \eqref{sav_tark_reformulation}, the resulting quadratic energy becomes a high-order approximation of the original energy. Although the SAV method may be criticized for this weakness, recent studies have attempted to overcome it. For example, Jiang et al. introduced the relaxed SAV approach in \cite{relaxation_sav} to connect the modified and original energy at a discrete level, \cite{relaxation_lag} proposed an alternating approach that combines the SAV and Lagrange multiplier methods to preserve the original energy. Our technique can also be utilized to develop higher-order schemes based on these approaches.
\end{rmk}
\begin{rmk}
    It should be noted that Theorem \ref{thm-energy_stable} guarantees the boundedness of the numerical solutions $\{u^n\}_{n=0}^{N_t}$ under the energy norm $\|\cdot\|_\mathcal{L}$, where $\|u\|_\mathcal{L} := (\mathcal{L}u, u)$. However, the solutions $\{v^n\}_{n=0}^{N_t}$ obtained from Algorithm \ref{ark_sav_alg} may not be bounded. Hence, Algorithm \ref{mark_sav_alg} is expected to be more stable in practical applications since it does not involve the update of $v^n$.
\end{rmk}

Let us now concentrate on the solvability of SAV-MDIARK methods. Notice that the proof for SAV-DIARK methods is similar, and we omit it here.

\begin{thm}\label{thm-solvability}
    Assume that the mobility matrix satisfies $\mathcal{G} = - \mathcal{B}^* \mathcal{B}$ and the RK coefficients $a_{ii} \geq 0$ in Algorithm \ref{mark_sav_intermediate}. The semi-discrete SAV-MDIARK scheme is then uniquely solvable when the time step is sufficiently small. Here, $\mathcal{B}$ is a linear operator, $\mathcal{B}^\star$ represents its adjoint.
\end{thm}
\begin{proof}
    Since we are considering the DIRK method, the scheme to solve the intermediate variable $v_{ni}$ can be reformulated as follows:
    \begin{equation*}
        v_{ni} = u^n + \tau a_{ii} \mathcal{G} \mathcal{L} v_{ni} + \tau\sum\limits_{j=0}^{i-1} (a_{ij} \dot{v}_{nj}^\mathcal{L} + \widehat{a}_{ij} \dot{v}_{nj}^\mathcal{N}).
    \end{equation*}
    Notably, we can solve the above system one by one for $i$ from 0 to $s-1$, where the only unknown in each step is $v_{ni}$. Combining the self-adjoint of $\mathcal{L}$ and the assumption to $\mathcal{G}$, it is readily to assert the decomposition $\mathcal{G}\mathcal{L} = -\mathcal{A}^\star \mathcal{A}$. Therefore, the solution of $v_{ni}$ can be regarded as the minimization of the convex functional defined by:
    \begin{equation*}
        \mathcal{S}[v] = \frac{1}{2} (\|v\|^2 + \tau a_{ii}\|\mathcal{A} v\|^2) - \big(u^n + \tau \sum\limits_{j=0}^{i-1} (a_{ij} \dot{v}_{nj}^\mathcal{L} + \widehat{a}_{ij} \dot{v}_{nj}^\mathcal{N}), v\big).
    \end{equation*}
    Therefore, the unique solvability of $v_{ni}$ is straightforward. Then, we prove the solvability of the system coupled by $u_{ni}$ and $q_{ni}$. Let $f_{ni} = \frac{\delta W}{\delta u}[v_{ni}] - \nabla \cdot \frac{\delta W}{\delta \nabla u_{ni}}[v_{ni}]$. Thanks to the factor that $q_{ni}$ is in dependent of space, it can be updated by 
    \begin{equation*}
    	q_{ni} = \frac{q^n + \tau \sum\limits_{j=0}^{i-1}a_{ij} \dot{q}_{nj} + \tau a_{ii} (\mathcal{A}f_{ni}, \mathcal{A} u^1_{ni} ) }{1 + 2\tau a_{ii} \| \mathcal{B} f_{ni} \|^2 - \tau a_{ii} (\mathcal{A}f_{ni}, \mathcal{A} u_{ni}^2) },
    \end{equation*}
 	where $u^1_{ni}$ and $u_{ni}^2$ are defined by 
 	\begin{equation*}
 	\begin{aligned}
 		u_{ni}^1 &= \mathop{ \rm argmin }_{u} \frac{1}{2} (\|v\|^2 + \tau a_{ii}\|\mathcal{A} v\|^2) - \tau (u^n + \sum\limits_{j=0}^{i-1} a_{ij} \dot{u}_{nj}, u ), \\ 
 		u_{ni}^2 &= \mathop{ \rm argmin }_{u} \frac{1}{2} (\|v\|^2 + \tau a_{ii}\|\mathcal{A} v\|^2) - 2\tau a_{ii} (\mathcal{G} f_{ni}, u ). \\ 
 	\end{aligned}
 	\end{equation*}
    Since the time step is supposed to be sufficiently small, the solvability of the system can be straightforward.
\end{proof}
\section{Theoretical analysis}
\subsection{Estimates of the global error}
In this section, we present global error estimates for the semi-discrete SAV-MARK methods. To simplify the presentation, we consider only the classical $L^2$ gradient flow, i.e., $\mathcal{G} = -1$, $\mathcal{L} = -\Delta$ and 
\begin{equation*}
    \mathcal{F}(u) = \frac{1}{2} \|\nabla u\|^2 + \int_\Omega F(u) d\mathbf{x}.
\end{equation*}
Without loss of generality, our subsequent analysis is based on the following assumptions $\mathcal{A}1$--$\mathcal{A}3$:
	\begin{itemize}
		\item[$\mathcal{A}1:$] The implicit component of the ARK method is algebraically and diagonally stable.

		\item[$\mathcal{A}2:$] The exact solution of the system is sufficiently smooth in both space and time.
		\item[$\mathcal{A}3:$] The nonlinearity $F(\cdot)$ is twice differentiable.
	\end{itemize}
The SAV-MARK scheme for the AC equation is given by
\begin{equation}\label{err_01}
	\begin{aligned}
        v_{ni} = u^n + \tau \sum\limits_{j = 0}^{s-1} (a_{ij} \Delta v_{nj} - 2 \widehat{a}_{ij} q_{nj} W^\prime(v_{nj})), \
		u_{ni} = u^n + \tau \sum\limits_{j=0}^{s-1} a_{ij} \dot{u}_{nj}, \\
		q_{ni} =  q^n + \tau \sum\limits_{j=0}^{s-1} a_{ij} \dot{q}_{nj}, \ 
		u^{n+1} = u^n + \sum\limits_{i=0}^{s-1} b_i \dot{u}_{ni}, \ q^{n+1} = q^n + \sum\limits_{i=0}^{s-1} b_i \dot{q}_{ni},
	\end{aligned}
\end{equation}
where 
\begin{equation*}
    \dot{u}_{ni} = \Delta u_{ni} - 2 q_{ni} W^\prime(v_{ni}) , \ \dot{q}_{ni} = (W^\prime(v_{ni}), \dot{u}_{ni}), \ W^\prime(u) = \frac{F^\prime(u)}{ 2 \sqrt{ \int_\Omega F(u) dx + C_0 } }.
\end{equation*}
The major obstacle in establishing error estimates for the SAV-MARK method	 is obtaining a prior $L^\infty$ bound for the intermediate stages $v_{ni}$. To address this issue, previous researches truncated the nonlinearity to a global Lipschitz function with compact support. This technique is reliable when the continuous solution is bounded, and the numerical solution is sufficiently close to it. Here, we will adopt a similar approach. Let $U(\mathbf{x}, t)$ be the exact solution to the $L^2$ gradient flow and $Q(t) = \sqrt{ \int_\Omega F(U(\mathbf{x}, t)) d\mathbf{x} + C }$. We define
\begin{equation}
    M_u = \|U(\mathbf{x}, t)\|_{C([0, T]; L^\infty (\Omega))}, \quad \dot{M}_u = \|\dot{U}(\mathbf{x}, t)\|_{C([0, T]; L^\infty(\Omega))}, \quad M_q = \max\limits_{0 \leq t \leq T}|Q(t)|.
\end{equation}
The constants provided above are well-defined by the assumption $\mathcal{A}2$ and the definition of $Q(t)$. We denote by $\mathcal{B} = M_u + 1$ and let
\begin{equation*}
    W^\prime_\mathcal{B}(s) = W^\prime(s) \rho(s/\mathcal{B}),
\end{equation*}
where $\rho(s)$ is a smooth function with compact support, such that
\begin{equation*}
    \rho(s) =  
    \left\lbrace 
    \begin{aligned}
        & 1, & 0 \leq |s| \leq 1, \\ 
        & \in [0, 1], & 1 \leq |s| \leq 2,  \\
        & 0, & |s| \geq 2.
    \end{aligned}
    \right.
\end{equation*}
It is readily to confirm that $W^\prime_\mathcal{B}(\cdot)$ is global Lipschitz continuous, and 
\begin{equation*}
\begin{aligned}
    & W^\prime_\mathcal{B} (s) = W^\prime(s), \quad \ \forall \ 0 \leq |s| \leq \mathcal{B}, \\ 
    & |W^\prime_\mathcal{B}(s)| \leq L_1, \quad | W^\prime_\mathcal{B} (r)  -  W^\prime_\mathcal{B} (s) | \leq L_2 |r - s|.
\end{aligned}
\end{equation*}
Following \cite{du_jsc_analysis,ju_jcp_if,tan_jcp_msrksav}, we introduce reference solutions $\mathcal{V}_{ni}$, $\mathcal{U}_{ni}$, $\mathcal{Q}_{ni}$, $\mathcal{U}^n$ and $\mathcal{Q}^n$, such that
\begin{equation}\label{err_02}
	\begin{aligned}
        \mathcal{V}_{ni} &= U(t_n) + \tau \sum\limits_{j = 0}^{s-1} (a_{ij} \Delta \mathcal{V}_{nj} - 2 \widehat{a}_{ij} \mathcal{Q}_{nj} W^\prime_\mathcal{B} (\mathcal{V}_{nj})), \\ 
        \mathcal{U}_{ni} &= U(t_n) + \tau \sum\limits_{j=0}^{s-1} a_{ij} \dot{\mathcal{U}}_{nj}, \ \dot{\mathcal{U}}_{ni} = \Delta \mathcal{U}_{nj} - 2 \mathcal{Q}_{nj} W^\prime_\mathcal{B}(\mathcal{V}_{nj}), \\ 
        \mathcal{Q}_{ni} &=  Q(t_n) + \tau \sum\limits_{j=0}^{s-1} a_{ij} \dot{\mathcal{Q}}_{nj}, \ \dot{\mathcal{Q}}_{ni} = (W^\prime_\mathcal{B}(\mathcal{V}_{ni}), \dot{\mathcal{U}}_{ni}). \\ 
	\end{aligned}
\end{equation}
These reference solutions play important roles in obtaining global estimates for the SAV-MARK methods. 
\begin{lem}\label{err_lem_01}
    Suppose that the time step satisfies 
    \begin{equation*}
    \tau \leq \min \{ (2c_2)^{-1},  (4c(c_3 + c_4))^{-1} \},
    \end{equation*}
 where the constants above will be specified in the subsequent derivations. We have the following estimates for the intermediate solutions  $\mathcal{V}_{ni}$ 
	\begin{equation*}
        \|\mathcal{V}_{ni}\|_{L^\infty} \leq M_u + \frac{1}{2}, \quad  0 \leq n \leq N_t, \ 0 \leq i \leq s-1.
	\end{equation*}
	Moreover,
	\begin{equation*}
	\begin{aligned}
		& \sum\limits_{i= 0}^{s-1} (|Q(t_{ni}) - \mathcal{Q}_{ni}| + \| U(t_{ni}) - \mathcal{V}_{ni} \| + \| U(t_{ni}) - \mathcal{U}_{ni} \|) \leq c_3 \tau^2, \\ 
		& \sum\limits_{i=0}^{s-1} \|\Delta (U(t_{ni}) - \mathcal{V}_{ni})\| \leq c_4 \tau.
	\end{aligned}
	\end{equation*}
\end{lem}
\begin{proof}
	Since $W^\prime_\mathcal{B}(U(t_{ni})) = W^\prime(U(t_{ni}))$, the exact solutions satisfy
	\begin{equation}\label{err_03}
		\begin{aligned}
			U(t_{ni}) &= U(t_n) + \tau \sum\limits_{j = 0}^{s-1} (a_{ij} \Delta U(t_{nj}) - 2 \widehat{a}_{ij} Q(t_{nj}) W^\prime_\mathcal{B} (U(t_{nj}))) + \eta^v_{ni}, \\ 
			U(t_{ni}) &= U(t_n) + \tau \sum\limits_{j=0}^{s-1} a_{ij} \dot{U}(t_{nj}) + \eta_{ni}^u, \ \dot{U}_{ni} = \Delta U(t_{ni}) - 2 Q(t_{ni}) W^\prime_\mathcal{B}(U(t_{ni})), \\ 
			Q(t_{ni}) &=  Q(t_n) + \tau \sum\limits_{j=0}^{s-1} a_{ij} \dot{Q}(t_{nj}) + \eta^q_{ni}, \ \dot{Q}(t_{ni}) = (W^\prime_\mathcal{B}(U(t_{ni})), \dot{U}(t_{ni})), \\ 
		\end{aligned}
	\end{equation}
	where 
    \begin{equation*}
        \sum\limits_{i=0}^{s-1} (\|\eta_{ni}^v\| + \|\eta_{ni}^u \| + |\eta_{ni}^q|) \leq c_1 \tau^{2}.
    \end{equation*}
    Subtracting the second and sixth equations of \eqref{err_03} from that of \eqref{err_02} yields
	\begin{equation}\label{err_04}
		\begin{aligned}
            U(t_{ni}) - \mathcal{V}_{ni} &= \tau \sum\limits_{j=0}^{s-1} ( a_{ij}  \Delta (U(t_{nj}) - \mathcal{V}_{nj}) - 2 \widehat{a}_{ij} \xi_{nj} ) + \eta_{ni}^v, \\ 
			U(t_{ni}) - \mathcal{U}_{ni} &= \tau \sum\limits_{j=0}^{s-1} a_{ij} (\dot{U}(t_{nj}) - \dot{\mathcal{U}}_{nj}) + \eta_{ni}^u, \\ 
			Q(t_{ni}) - \mathcal{Q}_{ni} &= \tau \sum\limits_{j=0}^{s-1} a_{ij} (\dot{Q}(t_{nj}) - \dot{\mathcal{Q}}_{nj}) + \eta_{ni}^q,
		\end{aligned}
	\end{equation}
	where
	\begin{equation*} \label{err_05}
		\begin{aligned}
			&\dot{U}(t_{ni}) - \dot{\mathcal{U}}_{ni} = \Delta (U(t_{ni}) - \mathcal{U}_{ni}) - 2 \xi_{ni}, \ \xi_{ni} = Q(t_{ni}) W^\prime_\mathcal{B}(U(t_{ni}))  - \mathcal{Q}_{ni} W^\prime_\mathcal{B} (\mathcal{V}_{ni}), \\ 
			&\dot{Q}(t_{ni}) - \dot{\mathcal{Q}}_{ni} = (W^\prime_\mathcal{B} (U(t_{ni})) - W^\prime_\mathcal{B}(\mathcal{V}_{ni}), \dot{U}(t_{ni})) + ( W^\prime_\mathcal{B}(\mathcal{V}_{ni}), \dot{U}(t_{ni}) - \dot{\mathcal{U}}_{ni} ).
		\end{aligned}
	\end{equation*}
	There is no difficulty in confirming that
	\begin{equation} \label{err_06}
		\begin{aligned}
			\|\xi_{ni}\| &\leq M_u \|U(t_{ni}) - \mathcal{V}_{ni}\| + L_1 |Q(t_{ni}) - \mathcal{Q}_{ni}|, \\
            |\dot{Q}(t_{ni}) - \dot{\mathcal{Q}}_{ni}| &\leq \dot{M}_u L_2 \|U(t_{ni}) - \mathcal{V}_{ni}\| + L_1 \|\dot{U}(t_{ni}) - \dot{\mathcal{U}}_{ni}\|.
		\end{aligned}
	\end{equation}
	According to Assumption $\mathcal{A}1$, there exists a positive definite diagonal matrix $\widetilde{H} = {\rm diag} \{ \widetilde{h}_0, \widetilde{h}_1, \cdots, \widetilde{h}_{s-1} \}$, such that $\widetilde{M} = \widetilde{H} A+ A^{\mathrm{T}} \widetilde{H}$ is positive definite. Therefore, we can find a sufficiently small constant $l$, such that
	\begin{equation*} 
		\widetilde{M}_l = (\widetilde{m}_{ij}^l) = A^{-\mathrm{T}} \widetilde{M} A^{-1} - 2 l \widetilde{H} = A^{-\mathrm{T}} \widetilde{H} + \widetilde{H} A^{-1} - 2 l \widetilde{H}
	\end{equation*}
	is positive definite. Moreover, let $\widetilde{M}_d = \widetilde{H} A^{-1}$, $\widetilde{M}_s = \widetilde{H} A^{-1} \widehat{A}$. Then,
	\begin{equation*}
		\begin{aligned}
			0 & \leq 2 l \sum\limits_{i=0}^{s-1} \widetilde{h}_i \|U(t_{ni}) - \mathcal{V}_{ni}\|^2 - 2\tau \sum\limits_{i=0}^{s-1} \widetilde{h}_i (\Delta ( U(t_{ni}) - \mathcal{V}_{ni} ), U(t_{ni}) - \mathcal{V}_{ni}) \\ 
			& = 2 \sum\limits_{i,j = 0}^{s-1} \widetilde{m}_{ij}^d (U(t_{ni}) - \mathcal{V}_{ni}, U(t_{nj}) - \mathcal{V}_{nj}) - \sum\limits_{i,j = 0}^{s-1} \widetilde{m}_{ij}^l (U(t_{ni}) - \mathcal{V}_{ni}, U(t_{nj}) - \mathcal{V}_{nj}) \\ 
			& \qquad - 2\tau \sum\limits_{i=0}^{s-1} \widetilde{h}_i (\Delta ( U(t_{ni}) - \mathcal{V}_{ni} ), U(t_{ni}) - \mathcal{V}_{ni}) \\ 
			& = 2 \sum\limits_{i,j = 0}^{s-1} \widetilde{m}_{ij}^d (U(t_{ni}) - \mathcal{V}_{ni}, \eta_{nj}^v) - \sum\limits_{i,j = 0}^{s-1} \widetilde{m}_{ij}^l (U(t_{ni}) - \mathcal{V}_{ni}, U(t_{nj}) - \mathcal{V}_{nj}) \\ 
			& \qquad - 4 \tau \sum\limits_{i,j = 0}^{s-1} \widetilde{m}^s_{ij} (U(t_{ni}) - \mathcal{V}_{ni}, \xi_{nj}) \\ 
            & \leq 2 \overline{\lambda}_d \sum\limits_{i=0}^{s-1} \|U(t_{ni}) - \mathcal{V}_{ni}\| \sum\limits_{i=0}^{s-1} \|\eta_{ni}^v\| - \underline{\lambda}_l \sum\limits_{i=0}^{s-1} \|U(t_{ni}) - \mathcal{V}_{ni}\|^2 \\ 
            & \qquad + 4 \overline{\lambda}_s (\dot{M}_u L_2 + L_1) \tau \sum\limits_{i=0}^{s-1} \| U(t_{ni}) - \mathcal{V}_{ni} \| ( \sum\limits_{i=0}^{s-1} |Q(t_{ni}) - \mathcal{Q}_{ni}| + \sum\limits_{i=0}^{s-1} \| U(t_{ni}) - \mathcal{V}_{ni} \| ),
		\end{aligned}
	\end{equation*}
    where $\underline{\lambda}_\alpha$ and $\overline{\lambda}_\alpha$, $\alpha = d,l,s,h$ are the maximum and minimum eigenvalues of $\widetilde{M}_d$, $\widetilde{M}_l$, $\widetilde{M}_s$, and $\widetilde{H}$, respectively. Consequently, 
	\begin{equation}\label{err_08}
        \sum\limits_{i=0}^{s-1} \|U(t_{ni}) - \mathcal{V}_{ni}\| \leq \frac{4 s \overline{\lambda}_s (\dot{M}_u L_2 + L_1)}{\underline{\lambda}_l} \tau \sum\limits_{i=0}^{s-1} ( \| U(t_{ni}) - \mathcal{V}_{ni} \| + |Q(t_{ni}) - \mathcal{Q}_{ni}| ) + \frac{2 s \overline{\lambda}_d}{\underline{\lambda}_l} \sum\limits_{i=0}^{s-1} \|\eta_{ni}^v\|.
	\end{equation}
	Following the same procedure, we can derive
	\begin{equation} \label{err_09}
        \sum\limits_{i=0}^{s-1} \|U(t_{ni}) - \mathcal{U}_{ni}\| \leq \frac{4 s \overline{\lambda}_h (\dot{M}_u L_2 + L_1)}{\underline{\lambda}_l} \tau \sum\limits_{i=0}^{s-1} ( \| U(t_{ni}) - \mathcal{V}_{ni} \| + |Q(t_{ni}) - \mathcal{Q}_{ni}| ) + \frac{2 s \overline{\lambda}_h}{\underline{\lambda}_l} \sum\limits_{i=0}^{s-1} \|\eta_{ni}^u\|.
	\end{equation}
	Combining \eqref{err_09} with the second equation of \eqref{err_04}, we have
	\begin{equation} \label{err_10}
	\begin{aligned}
        & \sum\limits_{i=0}^{s-1}\|\dot{U} (t_{ni}) - \dot{\mathcal{U}}_{ni} \| \leq \frac{4s \overline{\lambda}_d \overline{\lambda}_h (\dot{M}_u L_2 + L_1)}{\underline{\lambda}_l \underline{\lambda}_h} \sum\limits_{i=0}^{s-1} ( \| U(t_{ni}) - \mathcal{V}_{ni} \| + |Q(t_{ni}) - \mathcal{Q}_{ni}| ) \\   
        &\qquad \qquad + \frac{s \overline{\lambda}_d (2 \overline{\lambda}_h + \underline{\lambda}_l)}{\underline{\lambda}_l\underline{\lambda}_h}\tau^{-1} \sum\limits_{i=0}^{s-1} \|\eta_{ni}^u\|.
    \end{aligned}
	\end{equation}
	Subtracting the fourth equation of \eqref{err_03} with that of \eqref{err_02} gives 
	\begin{equation*}
		Q(t_{ni}) - \mathcal{Q}_{ni} = \tau \sum\limits_{j=0}^{s-1} a_{ij} (\dot{Q}(t_{ni}) - \dot{\mathcal{Q}}_{ni}) + \eta_{ni}^q.
	\end{equation*}
	Repeating to use the above technique and combining \eqref{err_06} and \eqref{err_10} then result in
	\begin{equation}\label{err_11}
		\begin{aligned}
			&\sum\limits_{i=0}^{s-1}|Q(t_{ni}) - \mathcal{Q}_{ni}| \leq \frac{2s \overline{\lambda}_h (\underline{\lambda}_h + 2s \overline{\lambda}_d \overline{\lambda}_h L_1)(\dot{M}_u L_2 + L_1) }{ \underline{\lambda}_l \underline{\lambda}_h} \tau \sum\limits_{i=0}^{s-1} (\| U(t_{ni}) - \mathcal{V}_{ni} \| + |Q(t_{ni}) - \mathcal{Q}_{ni}|) \\ 
			& \qquad \qquad + \frac{ 2s^2 \overline{\lambda}_h \overline{\lambda}_d (2 \overline{\lambda}_h + \underline{\lambda}_l) L_1 }{ \underline{\lambda}_l^2 \underline{\lambda}_h } \sum\limits_{i=0}^{s-1} (\|\eta_{ni}^u\| + |\eta_{ni}^q|).
		\end{aligned}
	\end{equation}
	Adding \eqref{err_08}, \eqref{err_09} and \eqref{err_11} together yields
	\begin{equation*}
		\begin{aligned}
			&\sum\limits_{i= 0}^{s-1} ( \| U(t_{ni}) - \mathcal{V}_{ni} \| + \| U(t_{ni}) - \mathcal{U}_{ni} \| + |Q(t_{ni}) - \mathcal{Q}_{ni}| ) \\ 
			&\qquad \leq c_2 \tau \sum\limits_{i= 0}^{s-1} ( \| U(t_{ni}) - \mathcal{V}_{ni} \| + \| U(t_{ni}) - \mathcal{U}_{ni} \| + |Q(t_{ni}) - \mathcal{Q}_{ni}| ) + \frac{c_3}{2} \tau^{2}.
		\end{aligned}
	\end{equation*}
	It follows by setting $\tau \leq (2c_2)^{-1}$ that
	\begin{equation}\label{err_12}
		\sum\limits_{i= 0}^{s-1} ( \| U(t_{ni}) - \mathcal{V}_{ni} \| + \| U(t_{ni}) - \mathcal{U}_{ni} \| + |Q(t_{ni}) - \mathcal{Q}_{ni}| ) \leq c_3 \tau^2.
	\end{equation}
	Using equation \eqref{err_12}, we then demonstrate the boundedness of $\mathcal{V}_{ni}$ for sufficiently small $\tau$. Inserting \eqref{err_12} into the first equation of \eqref{err_04} infers
	\begin{equation*}
        \sum\limits_{i=0}^{s-1} \|\Delta (U(t_{ni}) - \mathcal{V}_{ni})\| \leq \frac{\overline{\lambda}_d + 2 \overline{\lambda}_s (M_u + L_1)}{\underline{\lambda}_h \tau} ( \|U(t_{ni}) - \mathcal{V}_{ni}\| + |Q(t_{ni}) - \mathcal{Q}_{ni}| + \|\eta_{ni}^v\| ) \leq c_4 \tau.
	\end{equation*}
	The Sobolev inequality $\|f\|_{L^\infty} \leq c\|f\|_{H^2}$ and the triangular inequality give us
	\begin{equation*}
		\begin{aligned}
			\|\mathcal{V}_{ni}\|_{L^\infty} &\leq \|U(t_{ni})\|_{L^\infty} + \|U(t_{ni}) - \mathcal{V}_{ni}\|_{L^\infty} \\ 
			& \leq M_u + c \|U(t_{ni}) - \mathcal{V}_{ni}\|_{H^2} \leq M_u + 2c(c_3 + c_4) \tau.
		\end{aligned}
	\end{equation*}
	The estimate for $\mathcal{V}_{ni}$ in Lemma \ref{err_lem_01} is straightforward after setting $\tau \leq (4c(c_3 + c_4))^{-1}$. Therefore, we have completed the proof.
\end{proof}
Using the Taylor's formula and Lemma \ref{err_lem_01}, it is readily to confirm that when the time step satisfies the condition of Lemma \ref{err_lem_01}, the reference solutions further satisfy 
\begin{equation}\label{err_13}
	U(t_{n+1}) = U(t_n) + \tau \sum\limits_{i=0}^{s-1} b_i \dot{\mathcal{U}}_{ni} + \eta_{n+1}^u, \quad Q(t_{n+1}) = Q(t_n) + \tau \sum\limits_{i=0}^{s-1} b_i \dot{\mathcal{Q}}_{ni} + \eta_{n+1}^q,
\end{equation}
with
\begin{equation*}
	\|\eta_{n+1}^u\|_{H^1} + \|\eta_{n+1}^q\| \leq c_5 \tau^{p+1}.
\end{equation*}
We proceed to prove the convergence of the modified scheme obtained by replacing the nonlinear term $W^\prime(\cdot)$ in \eqref{err_01} with $W^\prime_\mathcal{B}(\cdot)$. For clarity, we remain to use the original notation to denote the solution of this modified scheme.
 Our proof demonstrates that $\|v_{ni}\|_{L^\infty} \leq M_u + 1$ for sufficiently small time steps. Consequently, $W^\prime_\mathcal{B}(v_{ni}) = W^\prime(v_{ni})$, which indirectly confirming the convergence of the SAV-MARK method \eqref{err_01}.

Let 
\begin{equation*}
	\mathcal{J}_{ni} = \mathcal{V}_{ni} - v_{ni}, \ \mathcal{E}_{ni} = \mathcal{U}_{ni} - u_{ni}, \ \mathcal{D}_{ni} = \mathcal{Q}_{ni} - q_{ni}.
\end{equation*}
Define solution errors 
\begin{equation*}
	E^{n+1} = U(t_{n+1}) - u^{n+1}, \ D^{n+1} = Q(t_{n+1}) - q^{n+1}.
\end{equation*}
\begin{thm} \label{thm-convergence}
	Let $c_\star =((3c_5^2 + c_{11})T\exp{(2c_{12}T)})^{\frac{1}{2}}$, and the time step 
	\begin{equation*}
		\tau \leq \min \{ (2c_2)^{-1},  (4c(c_3 + c_4))^{-1},  (2c_6)^{-1}, (4c(c_\star c_7 + c_8))^{-\frac{1}{p-1}}, (2c_{12})^{-1}\}.
	\end{equation*}
Then, the SAV-MARK method is convergent in the sense
	\begin{equation*}
		\|E^n\| + |D^n | \leq c_\star \tau^p, \quad 0 \leq n \leq N_t.
	\end{equation*}
\end{thm}
\begin{proof}
	We will complete the proof by the mathematical induction. As SAV-MARK is a one-step method, it is enough to prove the result for $n = l+1$ while assuming it holds for $n = l$. Let $n = l$. Subtracting \eqref{err_02} and \eqref{err_13} from \eqref{err_01}, we get
	\begin{equation}\label{err_14}
		\begin{aligned}
			\mathcal{J}_{li} &= E^l + \tau \sum\limits_{j=0}^{s-1} (a_{ij} \Delta \mathcal{J}_{lj} - 2 \widehat{a}_{ij} \zeta_{lj}), \\ 
			\mathcal{E}_{li} &= E^l + \tau \sum\limits_{j=0}^{s-1} a_{ij} \dot{\mathcal{E}}_{lj}, \ \mathcal{D}_{li} = D^l + \tau \sum\limits_{j=0}^{s-1} a_{ij} \dot{\mathcal{D}}_{lj}, \\ 
			E^{l+1} &= E^l + \tau\sum\limits_{i=0}^{s-1} b_i \dot{\mathcal{E}}_{li} + \eta_{l+1}^u, \ D^{l+1} = D^l + \tau\sum\limits_{i=0}^{s-1} b_i \dot{\mathcal{D}}_{li} + \eta_{l+1}^q,
		\end{aligned}
	\end{equation}
	where
	\begin{equation*}
		\begin{aligned}
			& \dot{\mathcal{E}}_{li} = \Delta \mathcal{E}_{li} - 2\zeta_{li}, \ \zeta_{li} = \mathcal{Q}_{li} (W^\prime_\mathcal{B}(\mathcal{V}_{li}) - W^\prime_\mathcal{B}(v_{li}) ) + \mathcal{D}_{li} W^\prime_\mathcal{B}(v_{li}), \\
			& \dot{\mathcal{D}}_{li} = (W^\prime_\mathcal{B} (\mathcal{V}_{li}) - W^\prime_\mathcal{B}(v_{li}), \dot{\mathcal{U}}_{li})  + (W^\prime_\mathcal{B} (v_{li}), \dot{\mathcal{E}}_{li}).
		\end{aligned}
	\end{equation*}
	Based on the proof of Lemma \ref{err_lem_01}, we can conclude that $|\mathcal{Q}_{li}| \leq \mathcal{M}_q$ and $|\dot{\mathcal{U}}_{li}| \leq \dot{\mathcal{M}}_u$. Applying the propositions of $W^\prime_\mathcal{B}(\cdot)$ then yields
	\begin{equation*}
		\|\zeta_{li}\|  \leq (\mathcal{M}_q L_2 + L_1) (\|\mathcal{J}_{li}\| + |\mathcal{D}_{li}| ), \quad |\dot{\mathcal{D}}_{li}| \leq (\dot{\mathcal{M}}_u L_2 + L_1)( \|\mathcal{J}_{li}\| + \|\dot{\mathcal{E}}_{li}\| ).
	\end{equation*}
	Furthermore, using \eqref{err_14} and the same technique employed in Lemma \ref{err_lem_01}, we can still arrive at
	\begin{equation*}
		\begin{aligned}
			\sum\limits_{i=0}^{s-1} \|\mathcal{J}_{li}\| &\leq  \frac{2 s \overline{\lambda}_s (\mathcal{M}_q L_2 + L_1)}{\underline{\lambda}_l} \tau \sum\limits_{i=0}^{s-1} ( \|\mathcal{J}_{li}\| + |\mathcal{D}_{li}| ) + \frac{2s \overline{\lambda}_d}{\underline{\lambda}_l}\|E^l\|, \\ 
			\sum\limits_{i=0}^{s-1} \|\mathcal{E}_{li}\| &\leq   \frac{2 s \overline{\lambda}_h (\mathcal{M}_q L_2 + L_1)}{\underline{\lambda}_l} \tau \sum\limits_{i=0}^{s-1} ( \|\mathcal{J}_{li}\| + |\mathcal{D}_{li}| ) +\frac{2s \overline{\lambda}_h}{\underline{\lambda}_l} \|E^l\|, \\ 
			\sum\limits_{i=0}^{s-1} |\mathcal{D}_{li}| &\leq  \frac{s \overline{\lambda}_h (\underline{\lambda}_d \underline{\lambda}_h + 2 s \overline{\lambda}_d \overline{\lambda}_h L_1)(\dot{\mathcal{M}}_u L_2 + L_1) }{\underline{\lambda}_l^2 \underline{\lambda}_h} \tau \sum\limits_{i=0}^{s-1} ( \|\mathcal{J}_{li}\| + |\mathcal{D}_{li}| ) \\
			&\quad + \frac{s^2 \overline{\lambda}_h \overline{\lambda}_d (2\overline{\lambda}_h + \underline{\lambda}_l) (\dot{\mathcal{M}}_u L_2 + L_1)}{\underline{\lambda}_l^2 \underline{\lambda}_h} |D^l|, \\ 
			\sum\limits_{i=0}^{s-1} \|\dot{\mathcal{E}}_{li}\| &\leq  \frac{2s \overline{\lambda}_d \overline{\lambda}_h(\mathcal{M}_q L_2 + L_1)}{\underline{\lambda}_l \underline{\lambda}_h} \sum\limits_{i=0}^{s-1} ( \|\mathcal{J}_{li}\| + |\mathcal{D}_{li}|  ) + \frac{s \overline{\lambda}_d (\underline{\lambda}_l + 2 \overline{\lambda}_h)}{\underline{\lambda}_l \underline{\lambda}_h} \tau^{-1} \|E^l\|.  
		\end{aligned}
	\end{equation*}
	Consequently,
	\begin{equation*}
			\sum\limits_{i=0}^{s-1} (  \|\mathcal{J}_{li}\| + \|\mathcal{E}_{li}\| + |\mathcal{D}_{li}| ) \leq c_6 \tau \sum\limits_{i=0}^{s-1} ( \|\mathcal{J}_{li}\| + \|\mathcal{E}_{li}\| + |\mathcal{D}_{li}| ) + \frac{c_7}{2}(\|E^l\| + |D^l|).
	\end{equation*}
	The restriction $\tau \leq (2c_6)^{-1}$ and the induction produce
	\begin{equation*}
		\sum\limits_{i=0}^{s-1} (  \|\mathcal{J}_{li}\| + \|\mathcal{E}_{li}\| + |\mathcal{D}_{li}| ) \leq c_\star c_7 \tau^p.
	\end{equation*} 
	Combining the above estimate with first equation of \eqref{err_14} then yields
	\begin{equation*}
		\|\Delta  \mathcal{J}_{li}\| \leq c_8 \tau^{p-1},
	\end{equation*}
	where $c_8 = \frac{\overline{\lambda}_d c_\star c_7 + s \overline{\lambda}_d c_\star + 2 \overline{\lambda}_s (\mathcal{M}_q L_2 + L_1)c_\star c_7}{\underline{\lambda}_h}$.
	Employing the inequalities $\|\nabla f\|^2 \leq \|f\|\|\Delta f\|$ and $\|f\|_{L^\infty} \leq c \|f\|_{H^2}$, it can be shown that if $\tau \leq (4c(c_\star c_7 + c_8))^{-\frac{1}{p-1}}$,
	\begin{equation*}
	\begin{aligned}
		\|v_{li} \|_{H^2} &\leq \|\mathcal{V}_{li}\|_{H^2} + \|\mathcal{J}_{li}\|_{H^2} \leq \|\mathcal{V}_{li}\|_{H^2} + 2(c_\star c_7 + c_8) \tau^{p-1} \leq c_9 , \\
		\|v_{li} \|_{L^\infty} &\leq \|\mathcal{V}_{li}\|_{L^\infty} + 2c(c_\star c_7 + c_8) \tau^{p-1} \leq M_u + 1.
	\end{aligned}
	\end{equation*}
	Let us now provide estimates for $E^{l+1}$ and $D^{l+1}$. Taking the difference between $\|E^{l+1}\|^2$ and $\|E^l\|^2$, and use the fourth equation of \eqref{err_14} yield
    \begin{equation}\label{diff_E}
		\begin{aligned}
			& \|E^{l+1}\|^2 - \|E^l\|^2 = 2\tau \sum\limits_{i=0}^{s-1} (E^l, b_i \dot{\mathcal{E}}_{li}) + \tau^2 \sum\limits_{i=0}^{s-1} \sum_{j=0}^{s-1} b_i b_j (\dot{\mathcal{E}}_{li}, \dot{\mathcal{E}}_{lj}) \\
			&\quad + 2 (E^l + \tau \sum\limits_{i=0}^{s-1} b_i \dot{\mathcal{E}}_{li}, \eta^u_{l+1}) + \|\eta_{l+1}^u\|^2.
		\end{aligned}
	\end{equation}
	Next, we individually estimate each of the terms on the right-hand side of \eqref{diff_E}. Based on the second equation of \eqref{err_14} and the algebraically stable condition, we deduce
	\begin{equation}\label{diff_E_rhs1}
		\begin{aligned}
			&2\tau \sum\limits_{i=0}^{s-1} (E^l, b_i \dot{\mathcal{E}}_{li}) + \tau^2 \sum\limits_{i=0}^{s-1} \sum_{j=0}^{s-1} b_i b_j (\dot{\mathcal{E}}_{li}, \dot{\mathcal{E}}_{lj}) + 2 \tau \sum\limits_{i=0}^{s-1} b_i \|\nabla \mathcal{E}_{li}\|^2 \\ 
			& = - \tau^2 \sum\limits_{i=0}^{s-1} \sum\limits_{j=0}^{s-1} m_{ij} (\dot{\mathcal{E}}_{li}, \dot{\mathcal{E}}_{lj}) + 2\tau \sum\limits_{i=0}^{s-1} b_i (\mathcal{E}_{li}, \dot{\mathcal{E}}_{li}) + 2 \tau \sum\limits_{i=0}^{s-1} b_i \|\nabla \mathcal{E}_{li}\|^2 \\ 
			& \qquad \leq   4(\mathcal{M}_q L_2 + L_1 + 1) \tau \sum\limits_{i=0}^{s-1} b_i (\|\mathcal{J}_{li}\|^2 + \|\mathcal{E}_{li}\|^2 + |\mathcal{D}_{li}|^2).
		\end{aligned}
	\end{equation}
	Using the Cauchy-Schwarz inequality and $ab \leq \frac{\tau}{2} a^2 + \frac{1}{2 \tau} b^2 $ yield
	\begin{equation}\label{diff_E_rhs2}
		\begin{aligned}
			&(E^l + \tau \sum\limits_{i=0}^{s-1} b_i \dot{\mathcal{E}}_{li}, \eta^u_{l+1}) = \|E^l\| \|\eta_{l+1}^u\| + \tau \sum\limits_{i=0}^{s-1} b_i (\Delta \mathcal{E}_{li} - 2\zeta_{li}, \eta^u_{l+1}) \\ 
			& \quad \leq \frac{\tau}{2} \|E^l\|^2 + \frac{1}{2\tau} \|\eta_{l+1}^u\|^2 + \tau \sum\limits_{i=0}^{s-1} b_i (\|\nabla \mathcal{E}_{li}\|\|\nabla \eta_{l+1}^u\| + 2 \|\zeta_{li}\|\|\eta_{l+1}^u\|) \\ 
			&\quad \leq \frac{\tau}{2} \|E^l\|^2 + \frac{\tau}{2} \sum\limits_{i=0}^{s-1} b_i \|\nabla \mathcal{E}_{li}\|^2 + 2(\mathcal{M}_q L_2 + L_1) \tau \sum\limits_{i=0}^{s-1} b_i (\|\mathcal{J}_{li}\|^2 + |\mathcal{D}_{li}|^2) + 2 c^2_5 \tau^{2p+1}.
		\end{aligned}
	\end{equation}
	Inserting \eqref{diff_E_rhs1} and \eqref{diff_E_rhs2} into \eqref{diff_E} infers
	\begin{equation}\label{err_gronwall_E}
		\begin{aligned}
			&\|E^{l+1}\|^2 + \tau \sum\limits_{i=0}^{s-1} b_i \|\nabla \mathcal{E}_{li}\|^2 \leq (1 + \tau )\|E^l\|^2  \\
			& \quad + 8(\mathcal{M}_q L_2 + L_1 + 1)\tau \sum\limits_{i=0}^{s-1} b_i (\|\mathcal{J}_{li}\|^2 + \|\mathcal{E}_{li}\|^2 + |\mathcal{D}_{li}|^2) + 3c_5^2 \tau^{2p+1}.
		\end{aligned}
	\end{equation}
	Analogously,
	\begin{equation}\label{err_gronwall_D}
		\begin{aligned}
			&|D^{l+1}|^2 \leq (1 + c_9 \tau) |D^l|^2 + c_{10}\tau \sum\limits_{i=0}^{s-1} (\|\mathcal{J}_{li}\|^2 + \|\mathcal{E}_{li}\|^2 + |\mathcal{D}_{li}|^2) + c_{11} \tau^{2p+1} .
		\end{aligned}
	\end{equation}
	Moreover, 
	\begin{equation}\label{err_gronwall_Rubbish}
		\begin{aligned}
			\sum\limits_{i=0}^{s-1} (\|\mathcal{J}_{li}\|^2 + \|\mathcal{E}_{li}\|^2 + |\mathcal{D}_{li}|^2)  &\leq ( \sum\limits_{i=0}^{s-1} \|\mathcal{J}_{li}\| + \|\mathcal{E}_{li}\| + |\mathcal{D}_{li}| )^2 \\ 
			&\leq 2 c_7^2 ( \|E^l\|^2 + |D^l|^2 ).
		\end{aligned}
	\end{equation}
	Collecting \eqref{err_gronwall_E}, \eqref{err_gronwall_D}, \eqref{err_gronwall_Rubbish} produces
	\begin{equation*}
		\|E^{l+1} \|^2 - |D^{l+1}|^2 \leq (1 + c_{12} \tau) (\|E^l\|^2 + |D^l|^2)  + (3c_5^2 + c_{11}) \tau^{2p+1}.
	\end{equation*}
	We observe that $c_{5}, c_{11}, c_{12}$ are independent of $c_\star$ and discrete parameters according to the derivations. By selecting $\tau \leq (2c_{12})^{-1}$ and applying the discrete Gronwall inequality, we can derive the desired result with $c_\star = ((3c_5^2 + c_{11})T\exp{(2c_{12}T)})^{\frac{1}{2}}$. Therefore, the proof is completed.
\end{proof}

\subsection{Relationships with the SAV-RK methods}\label{relationship}
In \cite{sav_nlsw}, Li et al. developed high-order unconditionally energy-stable schemes based on SAV techniques and RK methods. To obtain arbitrarily high-order and linearly implicit schemes, they proposed an iterative procedure to get a sufficiently accurate prediction of $u$, which was then used to discretize the nonlinear terms. In this section, we demonstrate that every SAV-RK methods can be viewed as an ARK method applied to some appropriate reformulations of \eqref{general_gradient_flow}. This new perspective enables us to systematically investigate the order conditions of existing works, utilizing the order conditions of ARK approaches. Employing their SAV-RKPC(M) methods to gradient flows leads to.  
\begin{alg}[SAV-RKPC(M)]\label{rkpc}
    Given a fundamental RK method with coefficients $(A, b, c)$, the intermediate variables are calculated by the prediction-correction procedure as

    1. Prediction: We initialize $u_{ni}^{(0)} = u^0, \ q_{ni}^{(0)} = q^0$. Let $M$ be a positive integer. Then, we iteratively compute $u_{ni}^{(m)}$ and $q_{ni}^{(m)}$ for $m = 0$ to $M-1$ by
    \begin{equation*}
    \left\lbrace
    \begin{aligned}
        &u_{ni}^{(m+1)} = u^n + \tau \sum\limits_{j=0}^{s-1} a_{ij} \dot{u}_{nj}^{(m+1)}, \ q_{ni}^{(m+1)} = q^n + \tau \sum\limits_{j=0}^{s-1}a_{ij} \dot{q}_{nj}^{(m+1)} \\
        & \dot{u}_{ni}^{(m+1)} = \mathcal{G} \big( \mathcal{L}u_{ni}^{(m+1)} + 2q_{ni}^{(m)} \frac{\delta W}{\delta u} [u_{ni}^{(m)}] - 2 q_{ni}^{(m)} \nabla \cdot \frac{\delta W}{\delta \nabla u}[u_{ni}^{(m)}] \big), \\ 
        & \dot{q}_{ni}^{(m+1)} = (\frac{\delta W}{\delta u} [u_{ni}^{(m+1)}], \dot{u}_{ni}^{(m+1)}) + (\frac{\delta W}{\delta \nabla u}[u_{ni}^{(m+1)}], \nabla \dot{u}_{ni}^{(m+1)}).
    \end{aligned}
    \right.
    \end{equation*}
    If $\max\limits_i \|u_{ni}^{(m+1)} - u_{ni}^{(m)}\|_\infty \leq TOL$, we stop the iterations and set $u_{ni}^{\star} = u_{ni}^{(m+1)}$. Otherwise, we set $u_{ni}^\star = u_{ni}^{(M)}$.

    2. Correction: For the predicted $u_{ni}^\star$, we compute the intermediate stages $\dot{u}_{ni}$ and $\dot{q}_{ni}$ as follows:
    \begin{equation*}
    \left\lbrace
    \begin{aligned}
       & u_{ni} = u^n + \tau \sum\limits_{j=0}^{s-1} a_{ij} \dot{u}_{nj}, \ q_i^n = q^n + \tau \sum\limits_{j=0}^{s-1} a_{ij} \dot{q}_{nj}, \\ 
       & \dot{u}_{ni} = \mathcal{G} ( \mathcal{L}u_{ni} + 2q_{ni} \frac{\delta W}{\delta u}[u_{ni}^\star] - 2q_{ni} \nabla \cdot \frac{\delta W}{\delta \nabla u}[u_{ni}^{\star}] ),\\
       &\dot{q}_{ni} = ( \frac{\delta W}{\delta u} [u_{ni}^\star], \dot{u}_{ni}) + ( \frac{\delta W}{\delta \nabla u} [u^\star_{ni}], \nabla \dot{u}_{ni} ),
    \end{aligned}
    \right.
    \end{equation*}
    and then update $u^{n+1}$, $q^{n+1}$ by
    \begin{equation*}
        u^{n+1} =u^n + \tau \sum\limits_{i=0}^{s-1} b_i \dot{u}_{ni}, \ q^{n+1} = q^n + \tau \sum\limits_{i=0}^{s-1} b_i \dot{q}_{ni}.
    \end{equation*}
\end{alg}
We display that Algorithm \ref{rkpc} can be regarded as an ARK method for the following alternative reformulation of \eqref{general_gradient_flow}.
\begin{equation} \label{sav_tark_reformulation2}
\left\lbrace
\begin{aligned}
    w_t &= \mathcal{G} \big( \mathcal{L}v + 2 r \frac{\delta W}{\delta u}[v] - 2 r \nabla \cdot \frac{\delta W}{\delta \nabla u}[v] \big), \\
    v_t &= \mathcal{G} \big( \mathcal{L}v + 2r \frac{\delta W}{\delta u}[v] - 2r \nabla \cdot \frac{\delta W}{\delta \nabla u}[v] \big), \\
    u_t &= \mathcal{G} \big( \mathcal{L}u + 2 q \frac{\delta W}{\delta u}[v] - 2q \nabla \cdot \frac{\delta W}{\delta \nabla u}[v]\big), \\
    r_t &= \big(\frac{\delta W}{\delta u}[v], v^{\mathcal{L}}_t\big) + \big(\frac{\delta W}{\delta u}[w], v^{\mathcal{\mathcal{N}}}_t\big) + \big( \frac{\delta W}{\delta \nabla u}[v], \nabla v_t^\mathcal{L} \big) + \big( \frac{\delta W}{\delta \nabla u}[w], \nabla v^\mathcal{N}_t \big), \\
    q_t &= \big( \frac{\delta W}{\delta u}[v], u_t \big) + \big( \frac{\delta W}{\delta \nabla u}[v], \nabla u_t \big),
\end{aligned}
\right.
\end{equation} 
where 
\begin{equation*}
    v_t^\mathcal{L} = \mathcal{G} \mathcal{L} v, \ v^\mathcal{N}_t = \mathcal{G} (2 r \frac{\delta W}{\delta u}[v] - 2 r \frac{\delta W}{\delta \nabla u}[v]).
\end{equation*}
Let us explain the equivalence between \eqref{sav_tark_reformulation2} and \eqref{general_gradient_flow}. Subtracting the second from the first equation of \eqref{sav_tark_reformulation2} and investigating the initial condition, we obtain:
\begin{equation*}
	v(t) = w(t), \quad \forall \ 0 < t \leq T.
\end{equation*}
Substituting this formula into the fourth equation of \eqref{sav_tark_reformulation2}, and subtracting the third equation of \eqref{sav_tark_reformulation2} from the second, the fifth equation of \eqref{sav_tark_reformulation2} from the fourth resulting in  
\begin{equation}\label{diff_eq2}
\begin{aligned}
    &u_t - v_t = \mathcal{G}\big(  \mathcal{L}(u - v) + 2(q - r) (\frac{\delta W}{\delta u}[v] - \nabla \cdot \frac{\delta W}{\delta \nabla u}[v]) \big), \\ 
    &q_t - r_t = \big( \frac{\delta W}{\delta u}[v], u_t - v_t\big) + \big( \frac{\delta W}{\delta \nabla u}[v], \nabla u_t - \nabla v_t \big).
\end{aligned}
\end{equation}
Taking the inner products on both sides of the first and the second equations in \eqref{diff_eq2} with $\mathcal{L}(u - v) + 2(q - r) (\frac{\delta W}{\delta u}[v] - \nabla \cdot \frac{\delta W}{\delta \nabla u}[v])$ and $2(q - r)$, respectively, and adding the resulting equations together yield:
\begin{equation*}
    \frac{1}{2} (u - v, \mathcal{L}(u - v)) + (q - r)^2 \leq 0.
\end{equation*}
This implies $u(t) = v(t), \ q(t) = r(t)$.  The remaining steps follow the proof of Lemma \ref{lem_equivalence}, which we omit here for brevity.

Let $\bm{z} = (w, v, u, r, q)^\mathrm{T}$. We split the reformulated system \eqref{sav_tark_reformulation2} as follows
\begin{equation}\label{three_partitioned}
    \bm{z}_t = \Phi_1(\bm{z}) + \Phi_2(\bm{z}) + \Phi_3(\bm{z}) + \Phi_4(\bm{z}), 
\end{equation}
where 
\begin{equation*}
    \Phi_1(\bm{z}) = 
    \left(
    \begin{smallmatrix}
        0 \\
        \mathcal{G} \mathcal{L} v \\
        \mathcal{G} ( \mathcal{L}u + 2 q \frac{\delta W}{\delta u}[v] - 2q \nabla \cdot \frac{\delta W}{\delta \nabla u}[v]) \\
        ( \frac{\delta W}{\delta u}[v], v^\mathcal{L}_t ) + ( \frac{\delta W}{\delta \nabla u}[v], \nabla v_t^\mathcal{L} ) \\
        ( \frac{\delta W}{\delta u}[v], u_t ) + ( \frac{\delta W}{\delta \nabla u}[v], \nabla u_t )
    \end{smallmatrix}
    \right), \quad
    \Phi_2(\bm{z}) = 
    \left(
    \begin{smallmatrix}
        0 \\
        \mathcal{G} ( 2 r \frac{\delta W}{\delta u}[v] - 2 r \nabla \cdot \frac{\delta W}{\delta \nabla u}[v]  ) \\ 
        0 \\
        ( \frac{\delta W}{\delta u}[v], w^\mathcal{N}_t ) + ( \frac{\delta W}{\delta \nabla u}[v], \nabla w_t^\mathcal{N} ) \\
        0
    \end{smallmatrix}
    \right),
\end{equation*}
\begin{equation*}
   \Phi_3(\bm{z}) = 
   \left(
    \begin{smallmatrix}
        \mathcal{G}\mathcal{L} v \\ 
        0 \\ 
        0 \\ 
        0 \\ 
        0
    \end{smallmatrix}
   \right),\quad
   \Phi_4(\bm{z}) =
   \left(
    \begin{smallmatrix}
        \mathcal{G} ( 2 r \frac{\delta W}{\delta \mathbf{u}}[\mathbf{v}] - 2 r \nabla \cdot \frac{\delta W}{\delta \nabla \mathbf{u}}[\mathbf{v}] ) \\ 
        0 \\ 
        0 \\ 
        0 \\ 
        0
    \end{smallmatrix}
   \right).
\end{equation*}
Employing four different RK methods to \eqref{three_partitioned} yields the following SAV-ARKII method
\begin{equation*}
\left\lbrace
\begin{aligned}
    \bm{z}_{ni} &= \bm{z}^n + \tau \sum\limits_{j=0}^{s-1} \big( a_{ij} \Phi_1(\bm{z}_{nj}) + \widehat{a}_{ij} \Phi_2(\bm{z}_{nj}) + \widetilde{a}_{ij} \Phi_3 (\bm{z}_{nj}) + \overline{a}_{ij} \Phi_4(\bm{z}_{nj}) \big), \\
    \bm{z}^{n+1} &= \bm{z}^n + \tau \sum\limits_{i=0}^{s-1} b_i (\Phi_1 (\bm{z}_{ni}) + \Phi_2(\bm{z}_{ni}) + \Phi_3 (\bm{z}_{ni}) + \Phi_4(\bm{z}_{ni}) ).
\end{aligned}
\right.
\end{equation*}
Furthermore, we rewrite the above scheme componentwisely and employ the techniques outlined in Section \ref{sec3-1} to modify the obtained scheme, ultimately resulting in the SAV-MARKII method as shown below.

\begin{alg}[SAV-MARKII]\label{mark2}
    We solve the intermediate stages from
    \begin{equation*}
    \left\lbrace
    \begin{aligned}
        &w_{ni} = u^n + \tau \sum\limits_{j=0}^{s-1} ( \widetilde{a}_{ij} \dot{v}_{nj}^\mathcal{L} + \overline{a}_{ij} \dot{v}^{\mathcal{N}}_{nj} ), \\
        &v_{ni} =  u^n + \tau \sum\limits_{j=0}^{s-1} (a_{ij} \dot{v}_{nj}^\mathcal{L} + \widehat{a}_{ij} \dot{v}^{\mathcal{N}}_{nj}), \ r_{ni} = r^n + \tau \sum\limits_{j=0}^{s-1} \big( a_{ij} \dot{r}_{nj}^\mathcal{L} + \widehat{a}_{ij} \dot{r}_{nj}^\mathcal{N} \big),  \\
        &u_{ni} = u^n + \tau \sum\limits_{j=0}^{s-1} a_{ij} \dot{u}_{nj}, \ q_{ni} = q^n + \tau \sum\limits_{j=0}^{s-1} a_{ij} \dot{q}_{nj}, \\ 
        & \dot{v}_{ni}^\mathcal{L} = \mathcal{G}\mathcal{L} v_{ni}, \ \dot{r}_{ni}^\mathcal{L} = ( \frac{\delta W}{\delta u}[v_{ni}], \dot{v}_{ni}^\mathcal{L}) + (\frac{\delta W}{\delta \nabla u}[v_{ni}], \nabla \dot{v}_{ni}^\mathcal{L}  ), \\ 
        & \dot{v}_{ni}^\mathcal{N} = \mathcal{G} ( 2r_{ni} \frac{\delta W}{\delta u}[v_{ni}] - 2r_{ni} \frac{\delta W}{\delta \nabla u}[v_{ni}] ) , \ \dot{r}_{ni}^\mathcal{N} = ( \frac{\delta W}{\delta u}[v_{ni}], \dot{v}_{ni}^\mathcal{N}) + (\frac{\delta W}{\delta \nabla u}[v_{ni}], \nabla \dot{v}_{ni}^\mathcal{N}  ), \\ 
        & \dot{u}_{ni} = \mathcal{G}(\mathcal{L} u_{ni} + 2q_{ni} \frac{\delta W}{\delta u}[v_{ni}] - 2q_{ni} \frac{\delta W}{\delta \nabla u}[v_{ni}]) , \ \dot{q}_{ni} = (\frac{\delta W}{\delta u} [v_{ni}], \dot{u}_{ni} ) + (\frac{\delta W}{\delta \nabla u} [v_{ni}], \nabla \dot{u}_{ni}). 
    \end{aligned}
    \right.
    \end{equation*}
    Then, we update 
    \begin{equation*}
        u^{n+1} = u^n + \tau \sum\limits_{i=0}^{s-1} b_i \dot{u}_{ni}, \ q^{n+1} = q^n + \tau \sum\limits_{i=0}^{s-1} b_i \dot{q}_{ni}.
    \end{equation*}
\end{alg}
\begin{thm}\label{thm_order_rkpc}
    Consider a SAV-RKPC(M) method associated with the fundamental RK method $(A, b, c)$ of stage $s$. Then, it can be regarded as a SAV-MARKII method with the tableaux
    \begin{equation} \label{arktab_grkpc}
    \begin{aligned}
        &\begin{array}{c | c }
            \mathbf{c} & \mathbf{A} \\ 
            \hline
                       & \mathbf{b}^\mathrm{T}
        \end{array}
        =
    \begin{array}{c | c c}
        \mathbf{0}           &   O & O  \\
        \mathbf{1}_M \otimes c & O  & I_M \otimes A \\ 
        \hline
                 & \mathbf{0}^\mathrm{T} &  (\mathbf{e}_M \otimes b)^\mathrm{T}
    \end{array}, \quad
        \begin{array}{c | c }
            \widehat{\mathbf{c}} & \widehat{\mathbf{A}} \\ 
            \hline
                                 & \widehat{\mathbf{b}}^\mathrm{T}
        \end{array}
        =
    \begin{array}{c | c c}
        \mathbf{0}           &  O  & O  \\
        \mathbf{1}_M \otimes c &  I_M \otimes A & O \\ 
        \hline
                               & \mathbf{0}^\mathrm{T} &  (\mathbf{e}_M \otimes b)^\mathrm{T}
    \end{array}, \\
    &\begin{array}{c | c }
            \widetilde{\mathbf{c}} & \widetilde{\mathbf{A}} \\ 
            \hline
                                 & \widetilde{\mathbf{b}}^\mathrm{T}
    \end{array}
        =
    \begin{array}{c | c c}
        \mathbf{1}_M \otimes c           &  O  & I_M \otimes A  \\
         c &  O & A \\ 
        \hline
                               & \mathbf{0}^\mathrm{T} &  (\mathbf{e}_M \otimes b)^\mathrm{T}
    \end{array}, \quad
    \begin{array}{c | c }
            \overline{\mathbf{c}} & \overline{\mathbf{A}} \\ 
            \hline
                                 & \overline{\mathbf{b}}^\mathrm{T}
    \end{array}
        =
    \begin{array}{c | c c}
        \mathbf{1}_M \otimes c           &  I_M \otimes A  & O  \\
         c &  A & O \\ 
        \hline
                               & \mathbf{0}^\mathrm{T} &  (\mathbf{e}_M \otimes b)^\mathrm{T}
    \end{array},
    \end{aligned}
    \end{equation}
  where $I_s$ represents the identity matrix, $\mathbf{e}_M = (0, 0, \cdots, 1)^{\mathrm{T}}$, and $\otimes$ denotes the Kronecker product. Notice that, we have $w_{n,i+ms} = v_{n,i+(m+1)s}$ and $r_{n,i+ms} = q_{n,i+(m+1)s}$ in Algorithm \ref{mark2}. In addition, the intermediate stages of Algorithm \ref{rkpc} and \ref{mark2} are related as follows:
    \begin{footnotesize}
    \begin{equation*}
        (\dot{u}_{ni}^{(m)}, \dot{q}_{ni}^{(m)}, q_{ni}^{(m)}, u_{ni}^{(m)}) = (\dot{v}^\mathcal{L}_{n,i+ms}+\dot{v}^\mathcal{N}_{n,i+ms},  \dot{\widehat{q}}_{n,i+ms}, \widehat{q}_{n,i+ms},  v_{n,i+ms}), \ u_{ni}^{\star} = v_{n,i+Ms}.
    \end{equation*}
\end{footnotesize}
\end{thm}
By Theorem \ref{thm_order_rkpc}, the consistency error of the SAV-RKPC(M) can be investigated by the order conditions of the generalized ARK methods straightforwardly. Readers are referred to \cite{ark_general} for convenience. Taking the fourth-order Gauss SAV-RKPC (SAV-GRK4PC) used in \cite{ieq_gong} as an example, the SAV-GRK4PC(1), SAV-GRK4PC(2), SAV-GRK4PC(3) methods arrive at second-, third- and fourth-order, respectively, which agrees with the numerical experiments proposed in \cite{sav_nlsw}.
\begin{rmk}
    Although we have demonstrated that the SAV-GRK4PC(3) achieves fourth-order accuracy, it is advisable to carry out additional iterative steps in practical computations to guarantee the stability of the proposed method.
\end{rmk}

\section{Numerical experiments} \label{sec5}
In this section, we demonstrate the effectiveness of our methods in solving the 2D AC, CH, and MBE equations. The spatial domain is $\Omega = (x_L, x_R) \times (y_L, y_R)$, and periodic boundary conditions are employed in all examples. To guarantee both accuracy and efficiency, we use the Fourier pseudo-spectral method for spatial discretization. Let $N_x$ and $N_y$ be positive integers. The spatial domain is uniformly partitioned with step sizes $h_x = \frac{x_R - x_L}{N_x}$ and $h_y = \frac{y_R - y_L}{N_y}$. We define $\Omega_N = \{ (x_i, y_j) |x_i = x_L + i h_x, \ y_j = y_L + j h_y \}$, and $\mathbb{M}_N$ denotes the space of periodic grid functions on $\Omega_N$. We use the notations $\nabla_N$, $\nabla_N \cdot$, and $\Delta_N$ to represent discrete gradient, divergence, and Laplace operators to the Fourier pseudo-spectral method, respectively. Readers are referred to \cite{ju_mbe} for details. Given $u, v \in \mathbb{M}_N$, the discrete $L^2$ inner product, discrete $L^2$ and $L^\infty$ norms are 
\begin{equation*}
    (u, v)_N = h_x h_y \sum\limits_{j=0}^{N_x-1}\sum\limits_{k=0}^{N_y-1} u_{jk} v_{jk}, \ \|u\|_N = \sqrt{(u, u)_N},  \ \|u\|_\infty = \max\limits_{0 \leq j \leq N_x-1 \atop 0 \leq k \leq N_y-1} |u_{jk}|.
\end{equation*}

\subsection{AC equation}
To validate the convergence results presented in Theorem \ref{thm-convergence}, we consider the following AC equation
\begin{equation*}
	u_t =  \varepsilon^2 \Delta u + u - u^3,
\end{equation*}
which can be obtained by setting $\mathcal{G} = -1$ and $\mathcal{F}[u] = \int_\Omega  \frac{\varepsilon^2}{2} |\nabla u|^2 + \frac{1}{4} (u^2 - 1)^2 d\mathbf{x}$ in \eqref{general_gradient_flow}. 

Employing the Fourier spectral method to \eqref{err_01}, the fully discrete system of the AC equation is to find $(\bm{u}_{ni}, \bm{v}_{ni}, q_{ni}) \in \mathbb{M}_N \times \mathbb{M}_N \times \mathbb{R}$ and $(\bm{u}^{n+1}, q^{n+1}) \in \mathbb{M}_N \times \mathbb{R}$, such that
\begin{equation*}
\begin{aligned}
    & \bm{v}_{ni} = \bm{u}^n +\tau \sum\limits_{j=0}^{s-1} (a_{ij} \Delta_N \bm{v}_{nj} - 2 \widehat{a}_{ij} q_{nj} W^\prime(\bm{v}_{nj})), \ \bm{u}_{ni} = \bm{u}^n + \tau \sum\limits_{j=0}^{s-1} a_{ij} \dot{\bm{u}}_{nj}, \\ 
    & q_{ni} = q^n + \tau \sum\limits_{j=0}^{s-1} a_{ij} \dot{q}_{nj}, \ \bm{u}^{n+1} = \bm{u}^n + \sum\limits_{i=0}^{s-1} b_i \dot{\bm{u}}_{ni}, \ q^{n+1} = q^n + \sum\limits_{i=0}^{s-1} b_i \dot{q}_{ni},
\end{aligned}
\end{equation*}
where 
\begin{equation*}
    \dot{\bm{u}}_{ni} = \Delta_N \bm{u}_{ni} - 2q_{ni} W^\prime (\bm{v}_{ni}), \ \dot{q}_{ni} = (W^\prime(\bm{v}_{ni}), \dot{\bm{u}}_{ni})_N, \ W^\prime(\bm{u}) = \frac{F^\prime(\bm{u})}{2\sqrt{ (F(\bm{u}), 1 )_N +C_0 }}.
\end{equation*}
\begin{rmk}
	It is worth mentioning that the discrete operator $\Delta_N$ satisfies the summation-by-parts formula. By following the procedure outlined in the proof of Theorem \ref{thm-energy_stable} and \ref{thm-solvability}, we can confirm the energy-stability and solvability of the above fully-discrete scheme.
\end{rmk}
\begin{figure}[H] 
	\centering
	\includegraphics[width=0.35\linewidth]{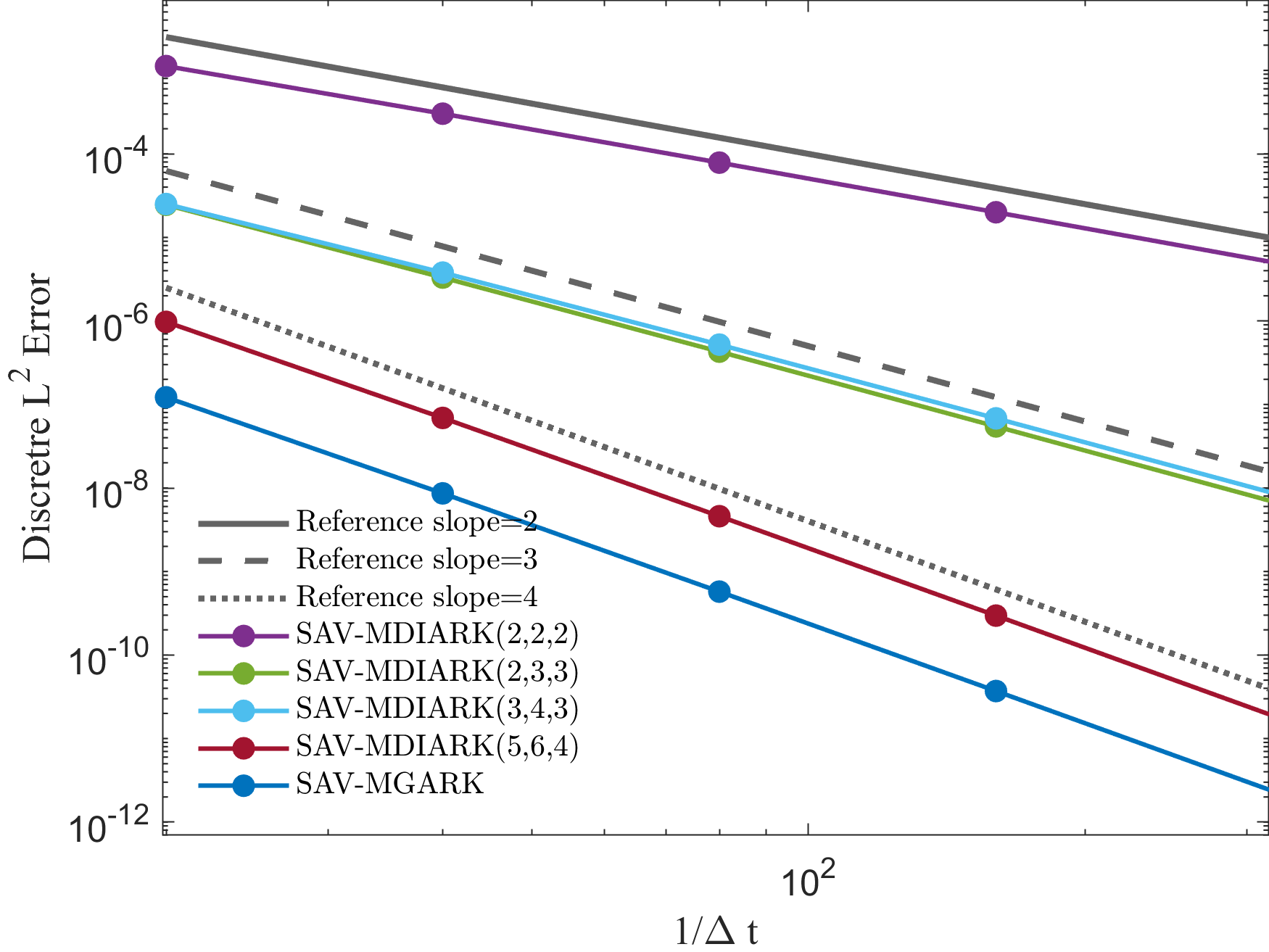}
	\caption{Temporal convergence tests of a various of different proposed methods.} \label{fig-ac_accuracy}
\end{figure}

We set the computational domain as $\Omega = (0, 1)^2$, the parameter as $\varepsilon = 0.01$, and the initial condition as $u_0 = 0.1 \sin (2 \pi x) \sin (2 \pi y)$. Since the exact solution is unavailable, we use the solution obtained by the SAV-MDIARK(5,6,4) method with $N = 512$, and $\tau = 10^{-4}$ at the final time $T = 1$ as a reference. Then, refinement test in time is conducted with $N = 128$ and different time steps $\tau = 0.1 \times 2^{-k} \ (k = 1,2,3,4,5)$. Figure \ref{fig-ac_accuracy} displays the discrete $L^2$-norm error of the solution at $T = 1$ computed by various methods as a function of the time step size in the logarithmic scale. All the methods achieve their respective accuracy.

\subsection{CH equation}
We consider the following Cahn-Hilliard model for immiscible binary fluids
\begin{equation}\label{cahn_hilliard}
    u_t =  \lambda \Delta (-\varepsilon^2 \Delta u + u^3 - u),
\end{equation}
where $\lambda$ is a mobility parameter, and $\varepsilon$ represents the width of the diffuse interface. The corresponding free energy functional is
\begin{equation}\label{fenergy_chan_hilliard}
    \mathcal{F}[u] = \int_\Omega  \frac{\varepsilon^2}{2}|\nabla u|^2 + \frac{1}{4}(u^2 - 1)^2 d\mathbf{x}.
\end{equation}
 We introduce an auxiliary variable $q = \sqrt{\frac{1}{4}\int_\Omega (u^2 - 1 -\kappa)^2 d\mathbf{x} + C|\Omega|}$, where $\kappa$ is a stabilized parameter. The energy functional \eqref{fenergy_chan_hilliard} is transformed into
\begin{equation}
    \mathcal{F}[u, q] = \frac{\varepsilon^2}{2} \|\nabla u\|^2 + \frac{\kappa}{2} \|u\|^2 + q^2 - \frac{\kappa^2 + 2\kappa + 4C}{4} |\Omega|.
\end{equation}
 \eqref{cahn_hilliard} is then reformulated into an equivalent model, as shown below
\begin{equation}\label{equivalent_cahn_hilliard}
\left\lbrace
\begin{aligned}
    u_t &= \lambda \Delta \left( -\varepsilon^2 \Delta u + \kappa u + f_\kappa(u) q\right), \\
    q_t &= \frac{1}{2}\left(f_\kappa(u), u_t\right). \\
\end{aligned}
\right.
\end{equation}

    We perform convergence tests in time by considering \eqref{cahn_hilliard} in the spatial domain $\Omega = (0, 2\pi)^2$ with specified parameters $\gamma = 0.01$ and $\varepsilon = 1$. As the exact solution of \eqref{cahn_hilliard} is not available, we construct a manufactured solution $\phi(x, y, t) = \sin(x)\sin{(y)}\cos{(t)}$ to \eqref{cahn_hilliard} by introducing a nonhomogeneous source term to the right-hand side of \eqref{cahn_hilliard}. We use $128 \times 128$ wave numbers for the spatial discretization. Subsequently, \eqref{cahn_hilliard} will be integrated using various methods until $T = 1$ with different time steps $\tau = 0.2 \times (2k)^{-1} \ (k = 1,2,3,4,5,6,7,8)$. The numerical solution at the final time is recorded to evaluate errors in the refinement tests.
\begin{figure}[H] 
    \centering
    \includegraphics[width=0.35\linewidth]{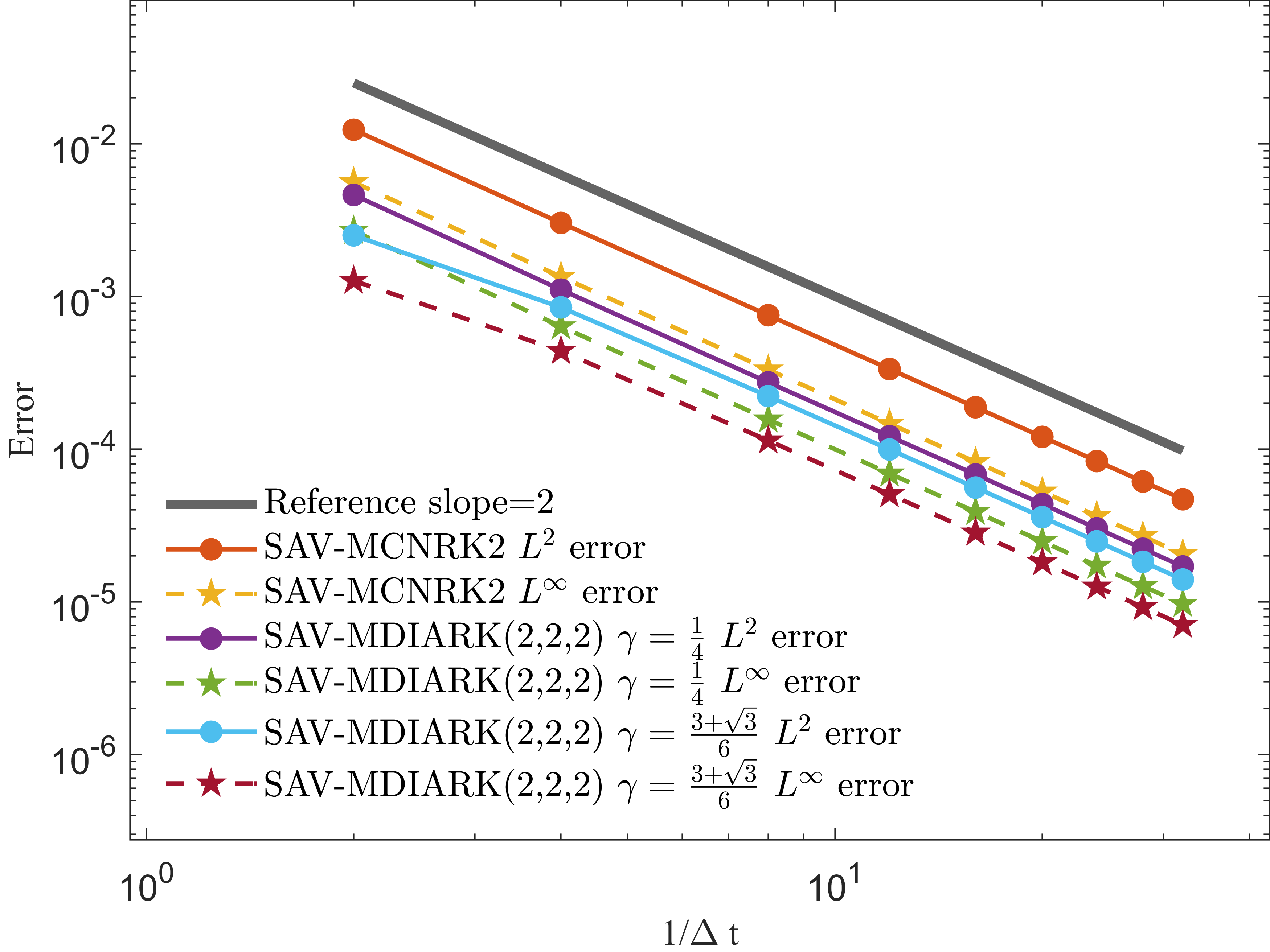} \quad
    \includegraphics[width=0.35\linewidth]{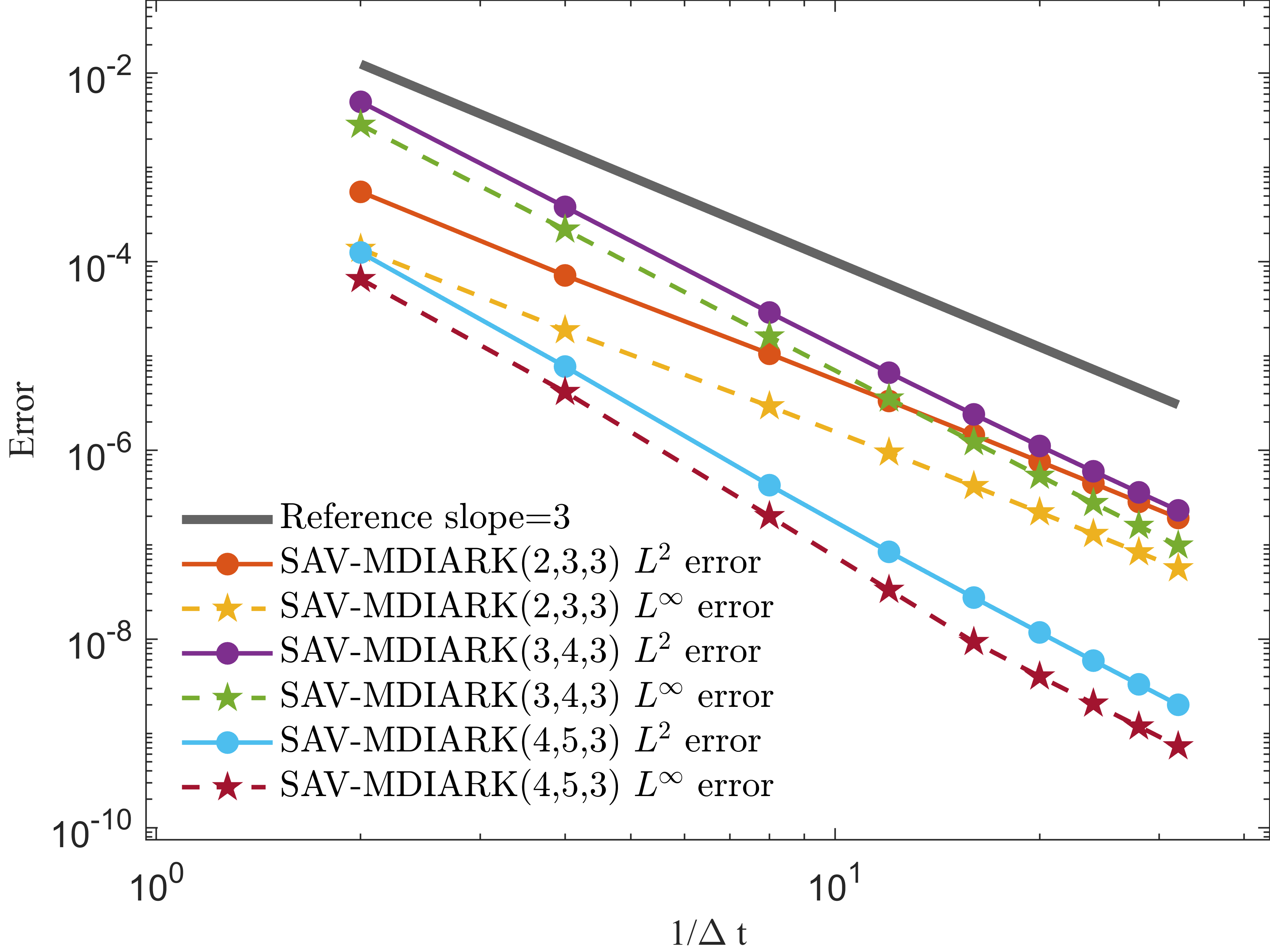} \quad
    \includegraphics[width=0.35\linewidth]{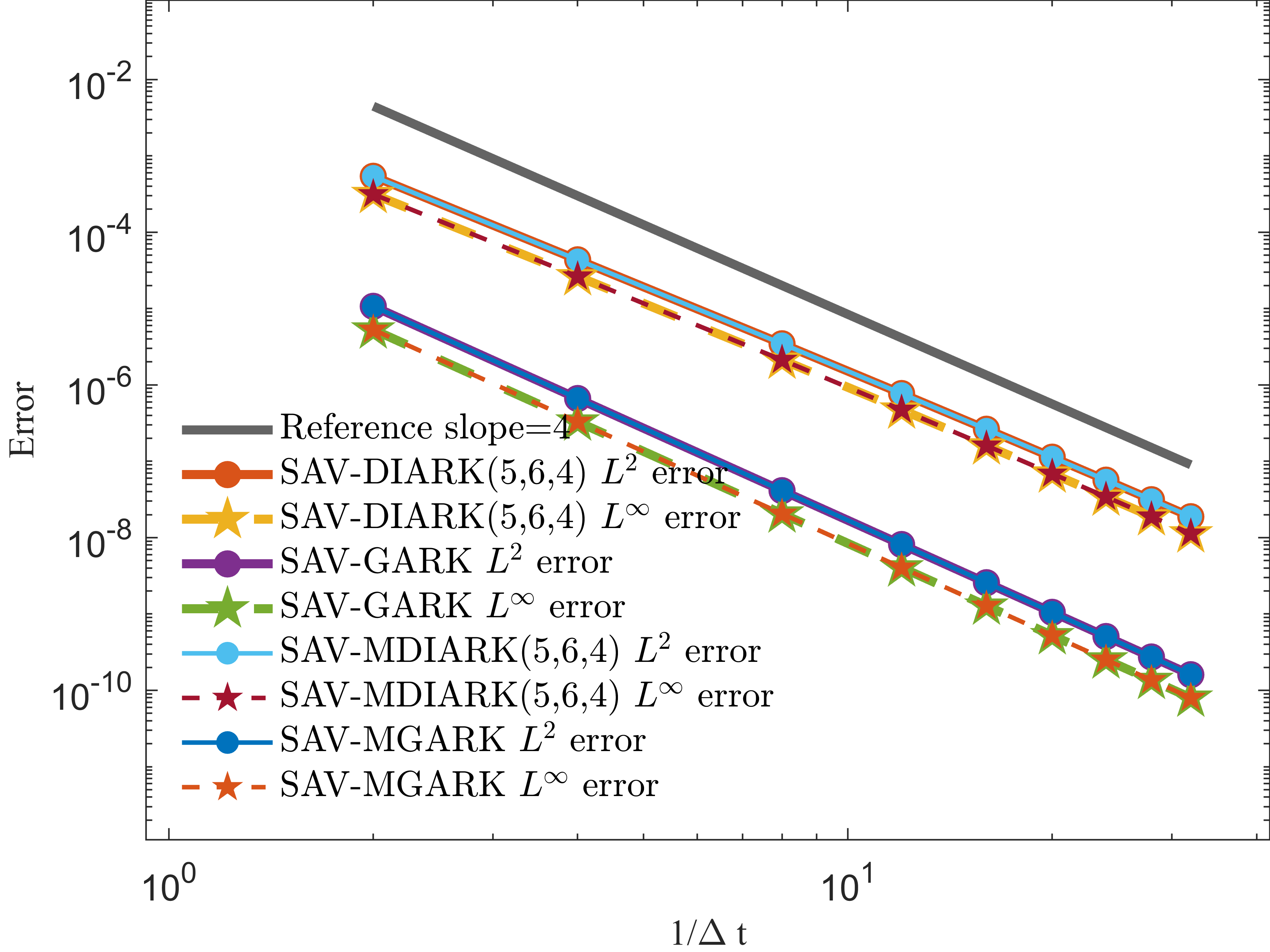}
    \caption{Temporal convergence tests of a various of different proposed methods.} \label{fig-ch_accuracy1}
\end{figure}
Figure \ref{fig-ch_accuracy1} plots the $L^2$ and $L^\infty$ errors of different methods against the time step in a logarithmic scale. All the methods achieve the expected convergence rate. Among the second-order schemes, the SAV-MDIARK(2,2,2) method exhibits higher accuracy than the SAV-MCNRK2 method. Additionally, when $\gamma = \frac{3 + \sqrt{3}}{6}$, the SAV-MDIARK(2,2,2) scheme performs to be better than when $\gamma = \frac{1}{4}$. Despite the latter preserving the dissipative rate, the former is more stable in practice. Among the third-order schemes, the SAV-MDIARK(4,5,3) exhibits the highest accuracy and unexpectedly results in superconvergence in this test. This phenomenon can be attributed to the smoothness of the provided solution. Further accuracy tests of this method will be conducted in subsequent examples. When investigating the fourth-order schemes, we present the results of both the SAV-MARK methods and their corresponding SAV-ARK methods. Notably, the convergence rate of the SAV-MARK methods is consistent with that of the SAV-ARK methods, confirming that the modified Algorithm \ref{mark_sav_alg} possesses the same accuracy as Algorithm \ref{ark_sav_alg}.

To thoroughly investigate the performance of the proposed schemes, we consider the CH equation \eqref{cahn_hilliard} with the initial condition 
\begin{equation}\label{eq-ch_cos_init}
    \phi_0(x, y) = 0.05 \big( \cos{(6 \pi x)} \cos{(8 \pi y)} + (\cos{(8 \pi x)} \cos{(6 \pi y)})^2 + \cos{(2 \pi x - 10 \pi y)}\cos{(4 \pi x - 2\pi y)} \big).
\end{equation}

We specify the spatial domain $\Omega = (0, 2\pi)^2$, and set the parameters in \eqref{cahn_hilliard} as $\lambda = 1$, $\varepsilon = 0.01$. The spatial discretization is carried out using $128 \times 128$ Fourier modes. Several methods are employed to solve the governing system until the final time $T = 0.1$. It should be noted that, due to the chosen initial condition \eqref{eq-ch_cos_init}, the solution of \eqref{cahn_hilliard} undergoes rapid changes at the beginning.  Therefore, if the method is not stable, it will fail to depict the solution using a large time step size accurately.
\begin{figure}[H]
    \centering
    \includegraphics[width=0.2\linewidth]{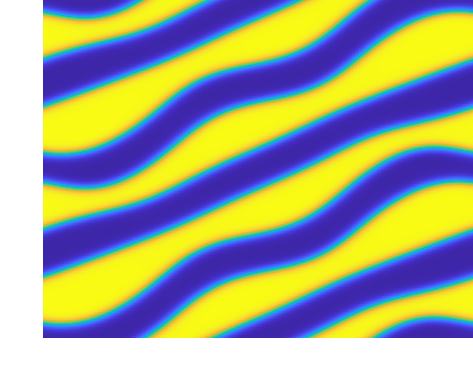} 
    \caption{Reference solution of the test problem solved by SAV-MDIARK(5,6,4) using the step size $\tau = 1\times 10^{-5}$ } \label{fig-reference_ch_cos}
\end{figure}
As a benchmark, Figure \ref{fig-reference_ch_cos} illustrates the snapshot obtained by the SAV-MDIARK(5,6,4) method with a step size of $\tau = 1 \times 10^{-5}$. During the test, the time step is progressively reduced until the correct solution snapshot is obtained, and the maximum step size that yields the correct solution profile for each method is recorded.
\begin{figure}[H]
    \centering
    \subfigure[\label{3a} $\tau=7.2\times 10^{-5}$ (left) and $\tau=6.25 \times 10^{-5}$ (right).]{
        \includegraphics[width=0.3\linewidth]{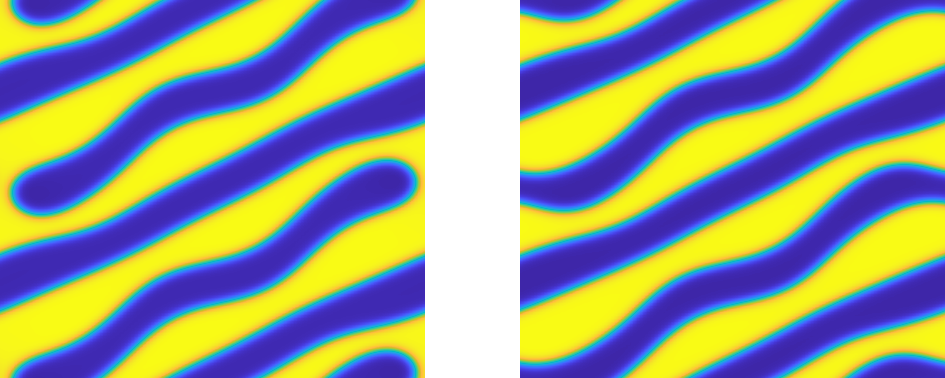}
    } \hspace{0.5cm}
    \subfigure[\label{3b} $\tau=1.25\times 10^{-4}$ (left) and $\tau=1 \times 10^{-4}$ (right).]{
        \includegraphics[width=0.3\linewidth]{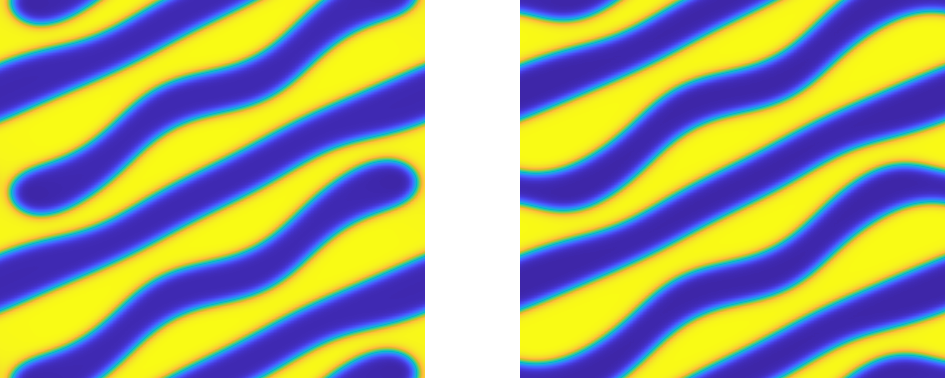}
    }
	\newline
    \subfigure[\label{3c} $\tau=4 \times 10^{-4}$ (left) and $\tau=3.75 \times 10^{-4}$ (right).]{
        \includegraphics[width=0.3\linewidth]{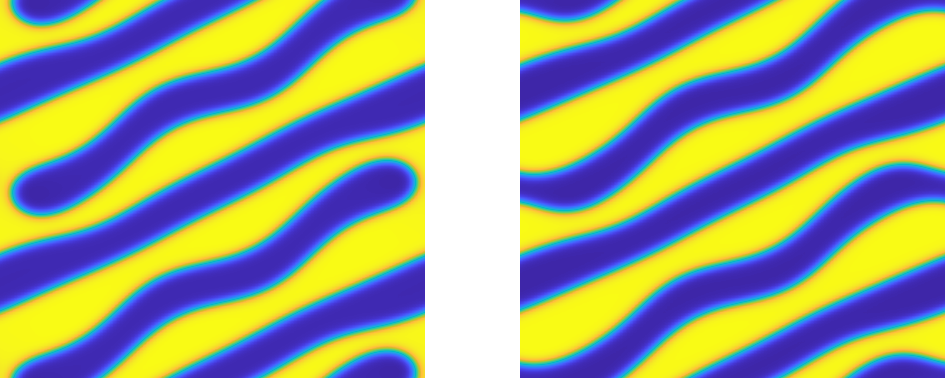}
    }
    \caption{Snapshots of $\phi$ solved by SAV-CN \subref{3a}, CS2 \subref{3b}, SAV-GRK4PC(5) \subref{3c} methods with different time step sizes. The stabilized parameters of SAV-CN and SAV-GRK4PC(5) are fixed $\kappa = 1$.} \label{fig-reference_ch_cos_other}
\end{figure}
To facilitate comparisons, we display numerical results for several existing methods in Figure \ref{fig-reference_ch_cos_other}, including the SAV-CN method, the fully implicit second-order convex splitting scheme (CS2), and the SAV-GRK4PC(5). It can be seen that the SAV-CN method fails to produce a correct result at a large time step, while the convex splitting scheme is capable of producing an accurate result with a relatively large time step. Due to the high precision and stability achieved through multiple iterations, the SAV-GRK4PC(5) can also compute a correct solution with a larger time step.
\begin{figure}[H]
    \centering
    \subfigure[\label{4a} $\tau=4.125 \times 10^{-5}$ (left) and $\tau=4 \times 10^{-4}$ (right).]{
        \includegraphics[width=0.3\linewidth]{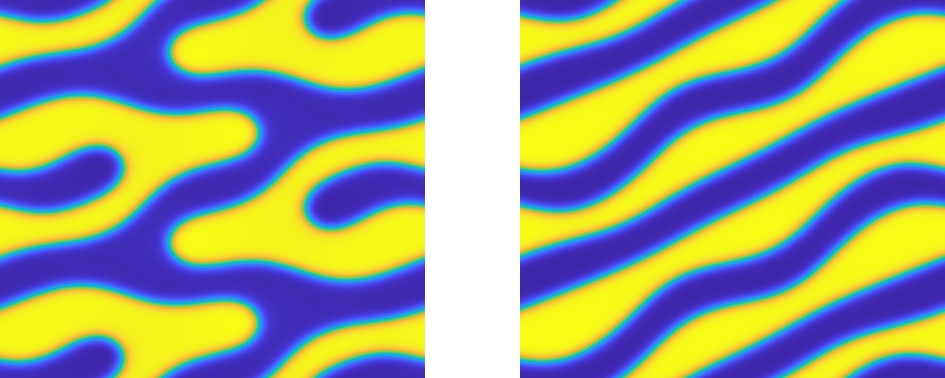}
    } \hspace{0.5cm}
    \subfigure[\label{4b} $\tau=4.2 \times 10^{-4}$ (left) and $\tau=4.125\times 10^{-4}$ (right).]{
        \includegraphics[width=0.3\linewidth]{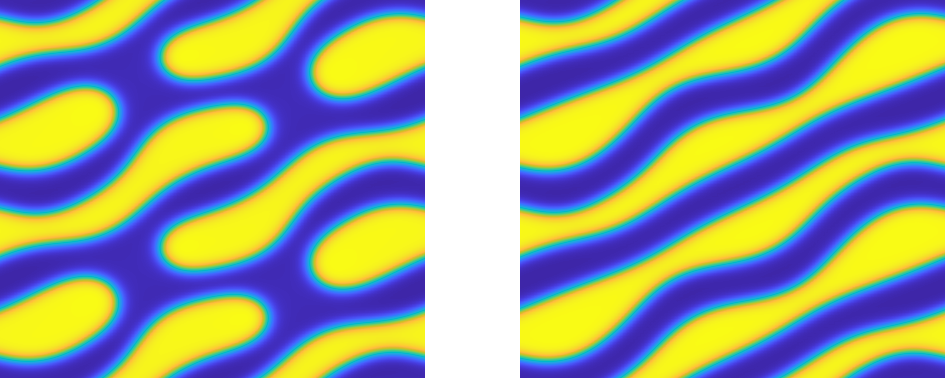}
    } 
\newline
    \subfigure[\label{4c} $\tau=5.25\times 10^{-4}$ (left) and $\tau = 5.2 \times 10^{-4}$ (right).]{
        \includegraphics[width=0.3\linewidth]{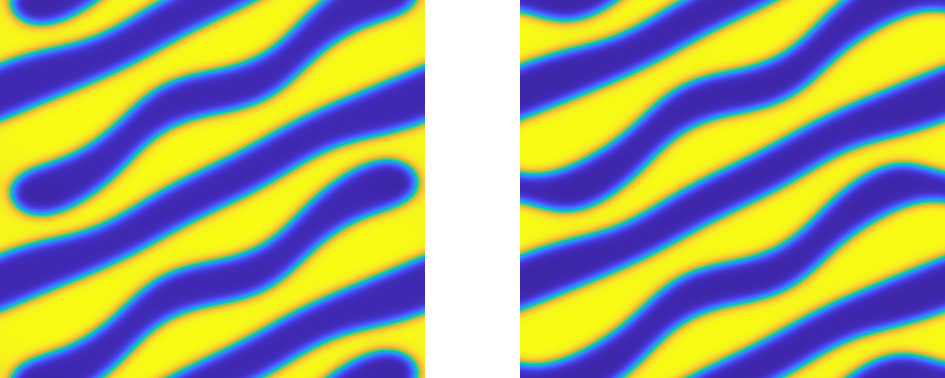}
    } \hspace{0.5cm}
    \subfigure[\label{4d} $\tau=3.75 \times 10^{-4}$ (left) and $\tau = 3.2 \times 10^{-4}$ (right).]{
        \includegraphics[width=0.3\linewidth]{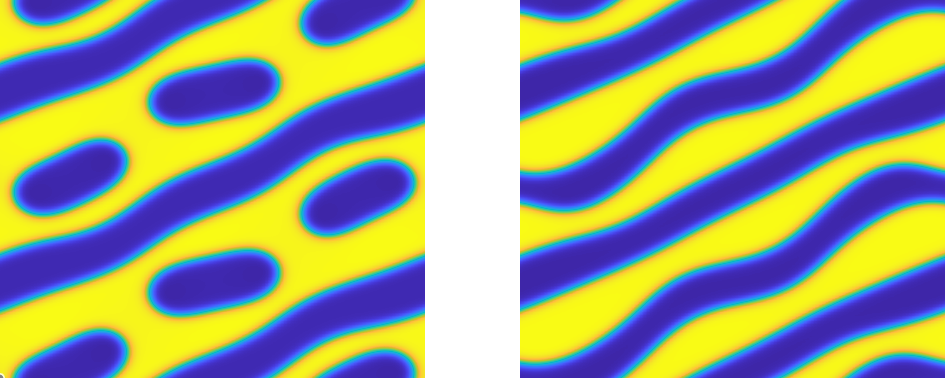}
    } \newline
    \subfigure[\label{4e} $\tau=2.875\times 10^{-4}$ (left) and $\tau = 2.5 \times 10^{-4}$ (right).]{
        \includegraphics[width=0.3\linewidth]{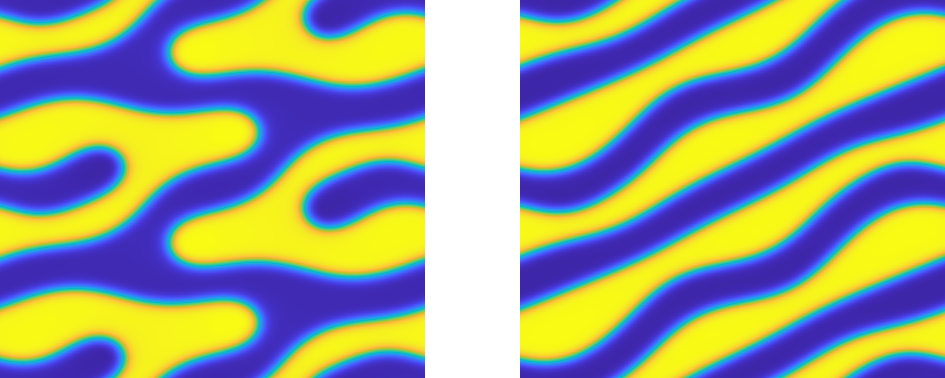}
    }
    \caption{Snapshots of $\phi$ solved by a various of different proposed methods SAV-MDIARK(2,2,2) \subref{4a}, SAV-MDIARK(2,3,3) \subref{4b}, SAV-MDIARK(4,5,3) \subref{4c}, SAV-MAGRK \subref{4d}, SAV-MDIARK(5,6,4) \subref{4e} with different time step sizes. The stabilized parameters of the proposed schemes are selected $\kappa = 0$.} \label{fig-reference_ch_cos_our}
\end{figure}
The numerical results obtained by the proposed schemes are presented in Figure \ref{fig-reference_ch_cos_our}. It is evident that our second- and third-order schemes achieve accurate results at larger step sizes compared to SAV-CN and CS2 methods and even outperform the SAV-GRK4PC(5) method. Among these methods, the SAV-MDIARK(5,4,3) method performs the best by yielding the correct solution at a step size of $\tau = 5.2 \times 10^{-4}$. Although the proposed fourth-order methods require smaller step sizes to obtain accurate results, their step sizes remain competitive with those used in other publications, despite considering only the order during their construction.

\begin{figure}[H]
    \centering
    \includegraphics[width=0.35\linewidth]{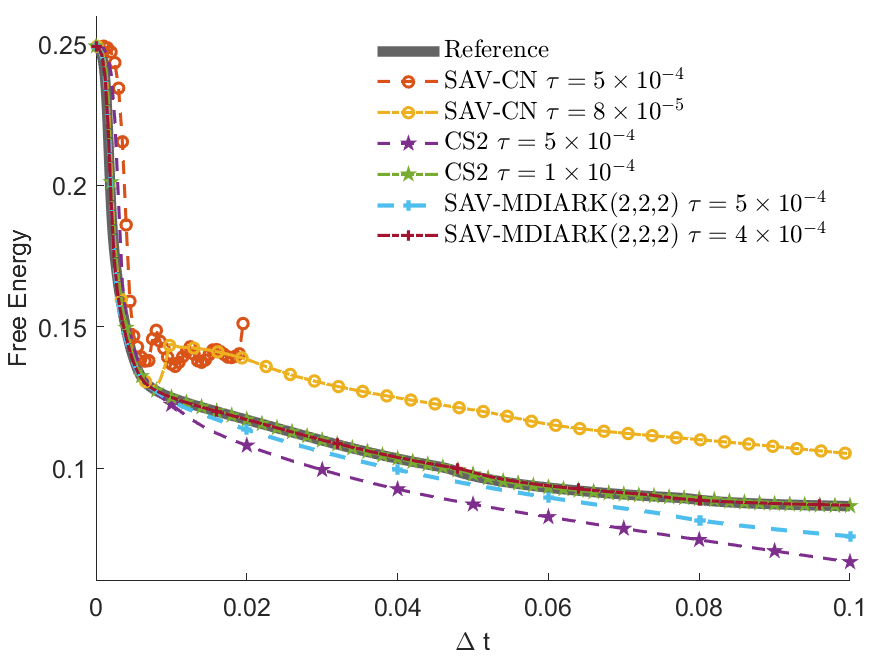} \quad
    \includegraphics[width=0.35\linewidth]{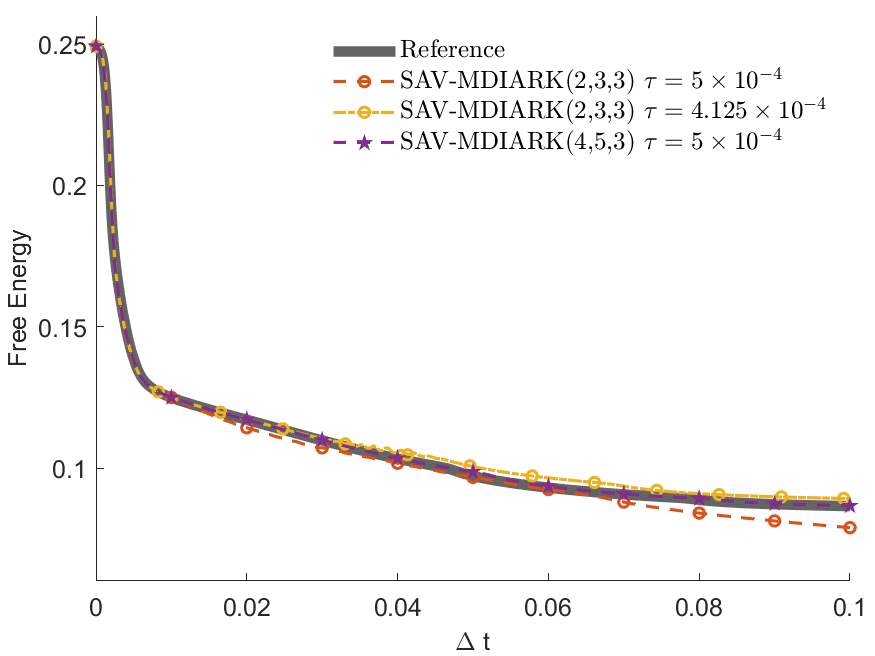} \quad
    \includegraphics[width=0.35\linewidth]{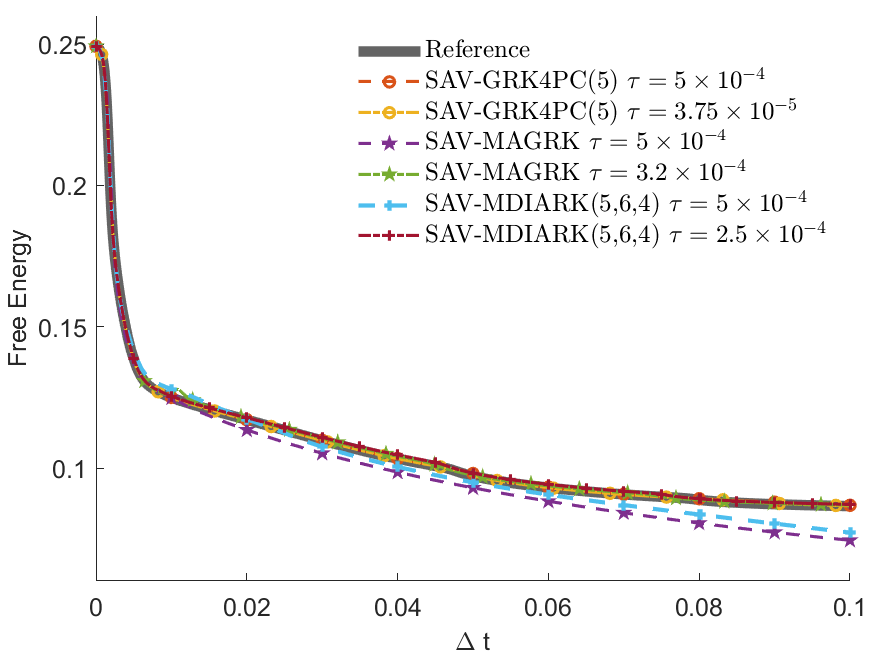}
    \caption{Time history of the free energy solved by a various of different proposed methods.} \label{ch-cos_fenergy}
\end{figure}
In addition to verifying the effectiveness of the proposed schemes through profiles and step sizes, we also present the evolution of the following discrete free energy
\begin{equation*}
    \mathcal{F}^n_N = \frac{\varepsilon^2}{2} \|\nabla_N \bm{u}^n\|_N^2 + \frac{1}{4} \big(((\bm{u}^n)^2 - 1)^2, 1\big)_N.
\end{equation*}
It is worth noting that although the above methods have only been proven to dissipate quadratic energy, we still investigate the original discrete energy in our experiments. Figure \ref{ch-cos_fenergy} summarizes the evolution of the free energy for different numerical schemes under different time steps. It can be observed that the SAV-CN method fails to dissipate the original energy at larger step sizes due to lower precision and weaker stability. While all our methods monotonically decrease the discrete free energy. This indicates that the proposed methods are robust and unconditionally energy-stable, as predicted by the theoretical results.

\subsection{MBE equation}
To further display the accuracy and robustness of the proposed schemes, let us consider the following MBE model  
\begin{equation}\label{mbe}
    u_t = - \lambda (\delta \Delta^2 u - \nabla\cdot f(\nabla u)),
\end{equation}
which is the $L^2$ gradient flow with respect to the following free energy functional
\begin{equation}\label{mbe fenergy}
    \mathcal{F}[u] = \int_\Omega \frac{\delta}{2} |\Delta u|^2 + F(\nabla u) d\mathbf{x}.
\end{equation}
In \eqref{mbe} and \eqref{mbe fenergy}, $u$ represents the height function of a thin film in a co-moving frame, $\delta$ is a positive constant, and $f = F^\prime$. If we set $F(\nabla u) = -\frac{1}{2} \ln(1 + |\nabla u|^2)$, \eqref{mbe} is usually called the MBE equation without slope selection. Corresponding, \eqref{mbe} is named MBE equation with slope selection while taking $F(\nabla u) = \frac{1}{4}(|\nabla u|^2 - 1)^2$. Introducing an SAV $q = \sqrt{\frac{1}{4} \int_\Omega (|\nabla u|^2 - 1 - \kappa)^2 d\mathbf{x} + C|\Omega|  }$. The free energy is then modified into
\begin{equation}
    \mathcal{F}[u, q] = \frac{\delta}{2}\|\Delta u\|^2 + \frac{\kappa}{2}\|\nabla u\|^2 + q^2 - \frac{\kappa^2 + 2\kappa + 4C}{4}|\Omega|.
\end{equation}
Correspondingly, \eqref{mbe} is reformulated into
\begin{equation}
    \left\lbrace
    \begin{aligned}
        u_t &= -\lambda \left( \delta \Delta^2 u - \kappa \Delta u - \nabla \cdot f_\kappa(\nabla u) q \right), \\ 
        q_t &= \frac{1}{2} (f_\kappa(\nabla u), \nabla u_t),
    \end{aligned}
    \right.
\end{equation}
where $f_\kappa(\nabla u) = \frac{(|\nabla u|^2 - 1 - \kappa)\nabla u}{\sqrt{ \frac{1}{4} \int_\Omega ( |\nabla u|^2 - 1 - \kappa )^2 d\mathbf{x} + C|\Omega| }}$.
\begin{rmk}
    We remark here although the nonlinearity of the MBE equation without slope selection seems to be unbounded, the SAV can remain to be introduced as $q = \sqrt{\frac{\kappa}{2}|\nabla u|^2 -\frac{1}{2} \ln(1 + |\nabla u|^2) +  C|\Omega|  }$. Due to the Lipschitz continuous of $F$, there is no difficulty in confirming that
\begin{equation*}
    \frac{\kappa}{2}|\nabla u|^2 -\frac{1}{2} \ln(1 + |\nabla u|^2) > 0,
\end{equation*}
as soon as $\kappa \geq \frac{1}{8}$.
\end{rmk}
We will still begin with performing the refinement test in time. Specifying the computational domain $\Omega = (0, 2\pi)^2$ and considering a classical example with the initial condition 
\begin{equation*}
    \phi_0(x, y) = 0.1(\sin{3x} \sin{5y} + \sin{5x} \sin{5y}),
\end{equation*}
which was studied in \cite{mbe_model1, mbe_model2} to observe morphological instability due to the nonlinear interaction. The parameters are $\lambda = 1$ and $\delta = 0.1$. Since the exact solution of \eqref{mbe} is not available, the SAV-MDIARK(5,6,4) method is employed to compute a reference solution of \eqref{mbe} using $256 \times 256$ Fourier modes and a step size of $\tau = 5 \times 10^{-6}$. Then the refinement test in time is carried out by varying the temporal step size $\tau = 2^{3-k} \times 10^{-4}$ ($k = 0,1,\cdots,6$). The spatial is discretized using $128 \times 128$ Fourier modes. The discrete $L^2$ and $L^\infty$ errors between the reference and numerical solutions at $T = 0.1$ are recorded.
\begin{figure}[H]
    \centering
    \includegraphics[width=0.35\linewidth]{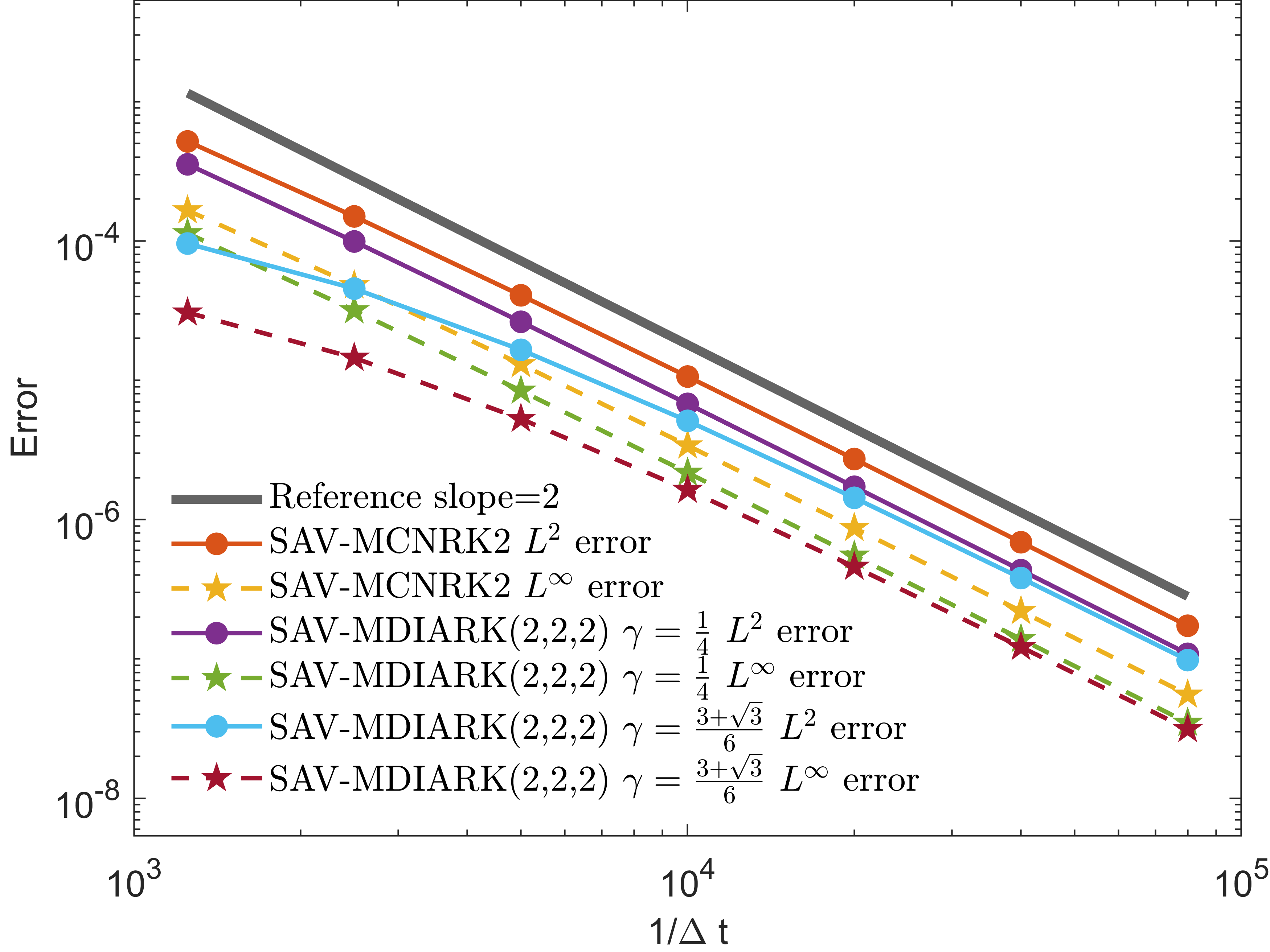} \quad
    \includegraphics[width=0.35\linewidth]{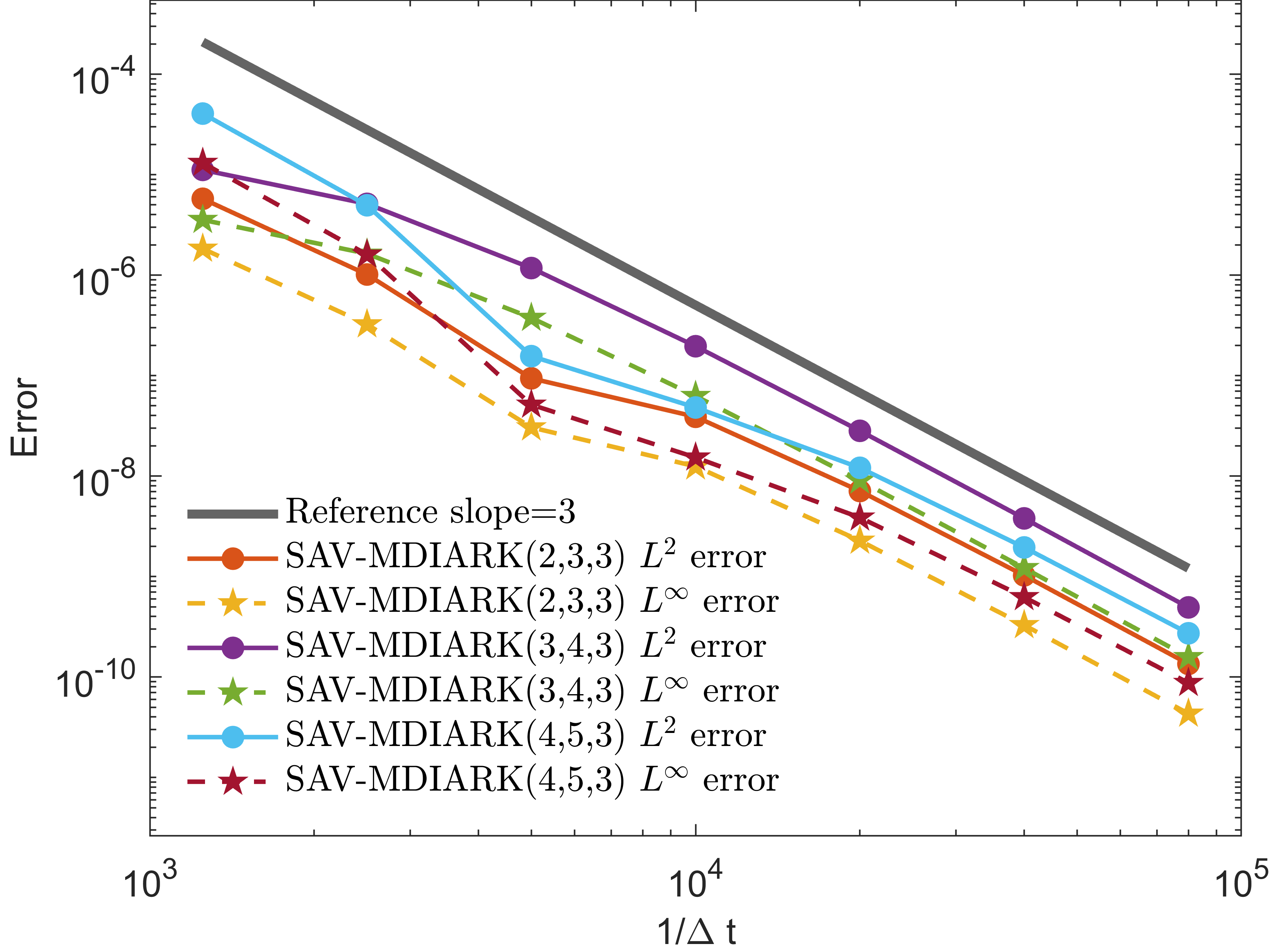} \quad
    \includegraphics[width=0.35\linewidth]{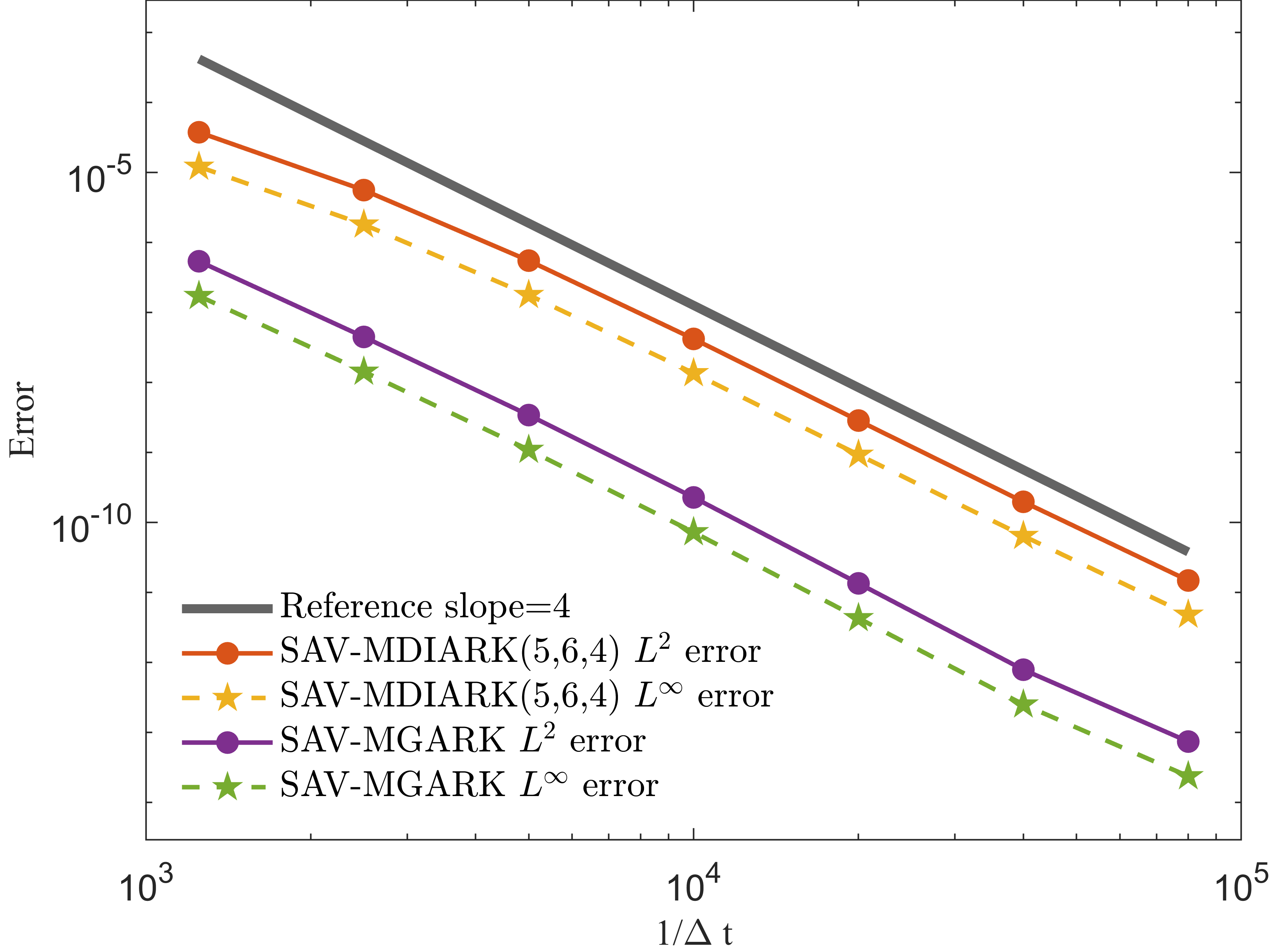}
    \caption{Temporal convergence tests of a various of different proposed methods.}\label{fig:mbe_conv}
\end{figure}
Figure \ref{fig:mbe_conv} displays the solution error at $T=0.1$ as a function of the step size in the logarithmic scale. It is observable that all methods arrive at their corresponding convergence rates. Moreover, the super-convergence of SAV-MDIARK(4,5,2) disappears under this circumstance, suggesting that the results appearing in Figure \ref{fig-ch_accuracy1} are coincidental.
\begin{figure}[H]
    \centering
    \subfigure[$t=0$]{
        \includegraphics[width=0.15\linewidth]{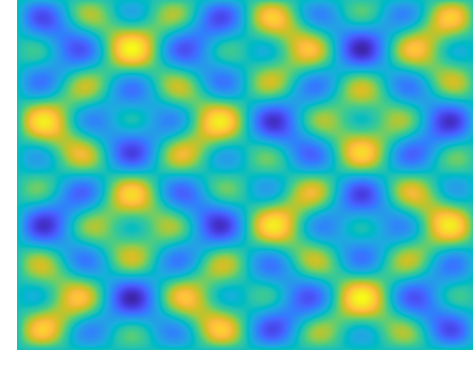}
    }
    \subfigure[$t=0.05$]{
        \includegraphics[width=0.15\linewidth]{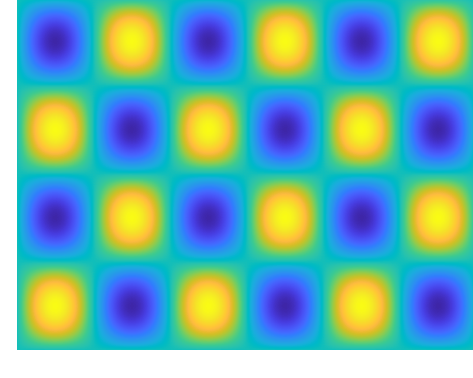}
    }
    \subfigure[$t=2.5$]{
        \includegraphics[width=0.15\linewidth]{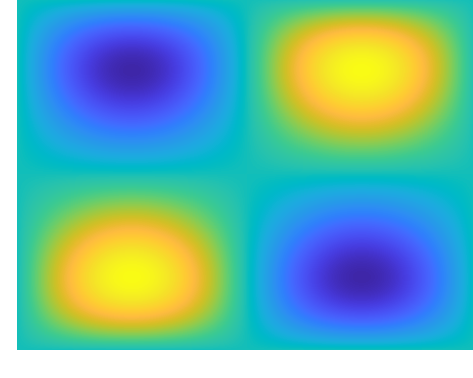}
    }

    \subfigure[$t=5.5$]{
        \includegraphics[width=0.15\linewidth]{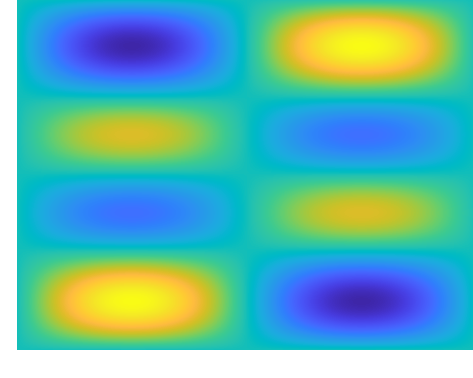}
    }
    \subfigure[$t=8$]{
        \includegraphics[width=0.15\linewidth]{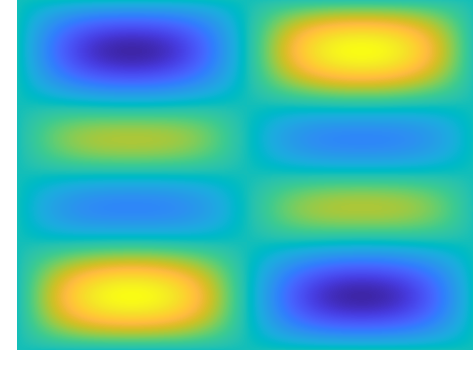}
    }
    \subfigure[$t=30$]{
        \includegraphics[width=0.15\linewidth]{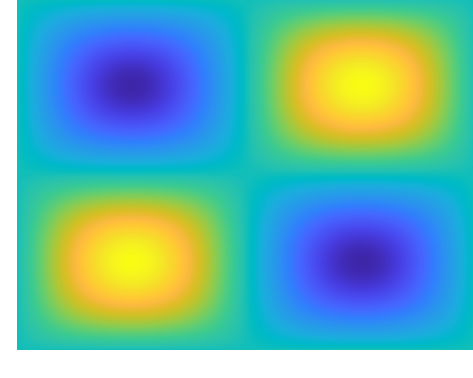}
    }
    \caption{Snapshots of $\phi$ at different times solved by the SAV-MDIARK(5,6,4) with the time step size $\tau = 5\times 10^{-3}$} \label{fig:mbe_height1}
\end{figure}
Then, we simulate \eqref{mbe} under the same initial condition until $T=30$. Figure \ref{fig:mbe_height1} displays the height profiles solved by the SAV-MDIARK(5,6,4) under with $\tau = 5 \times 10^{-3}$ at different times $t=0,0.05,2.5,5.5,8,30$. The results agree with those reported in \cite{tang_splitting,lu_lsp}. We remark here that the simulation results of other schemes are indistinguishable and thus are omitted due to the limitation of space.

\begin{figure}[H]
    \centering
    \subfigure[BDF2]{
    \includegraphics[width=0.28\linewidth]{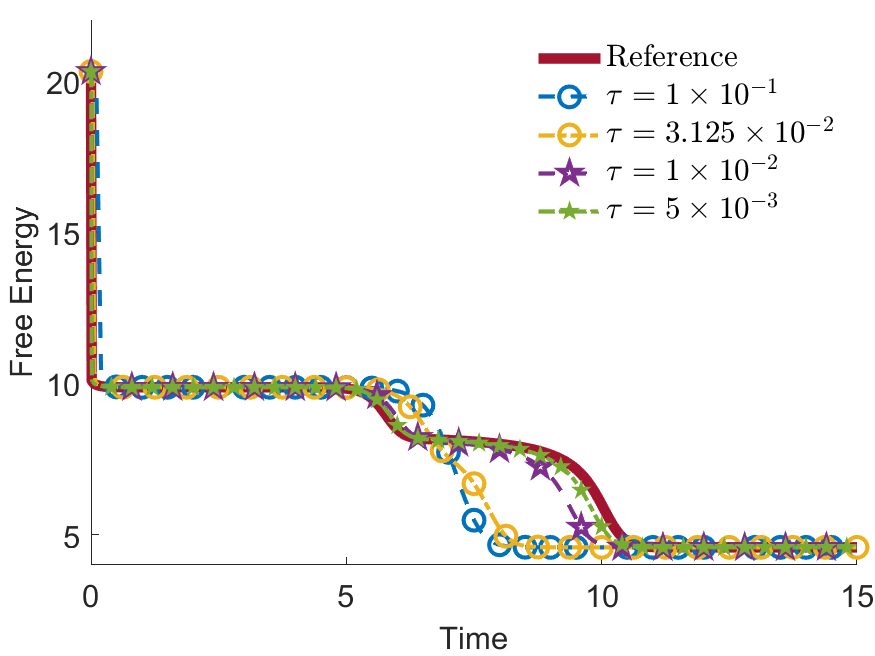} \quad
}
\subfigure[SAV-MDIARK(2,2,2) $\gamma = \frac{1}{4}$]{
    \includegraphics[width=0.28\linewidth]{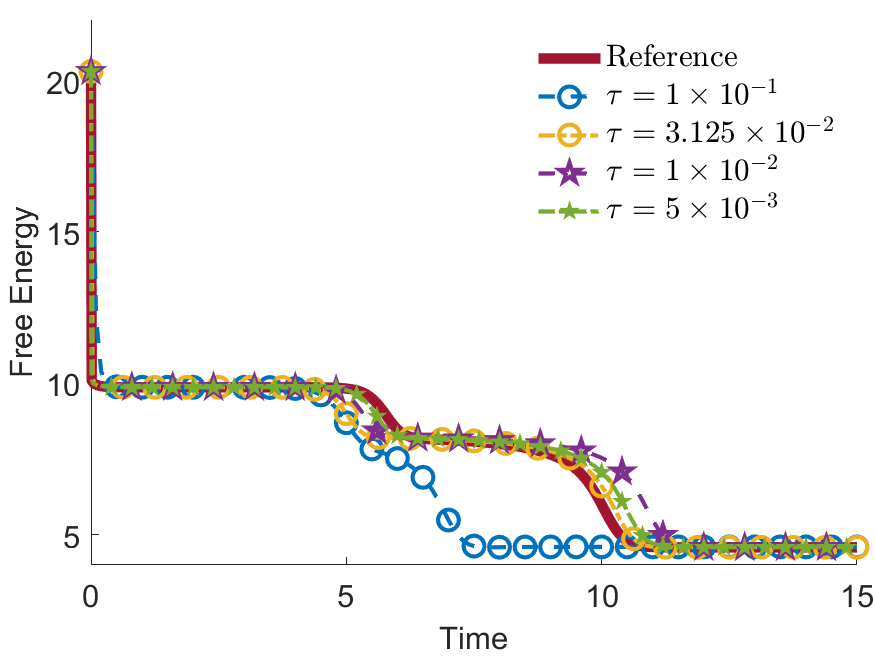} \quad
}
\subfigure[SAV-MDIARK(2,2,2) $\gamma = \frac{3+\sqrt{3}}{6}$]{
    \includegraphics[width=0.28\linewidth]{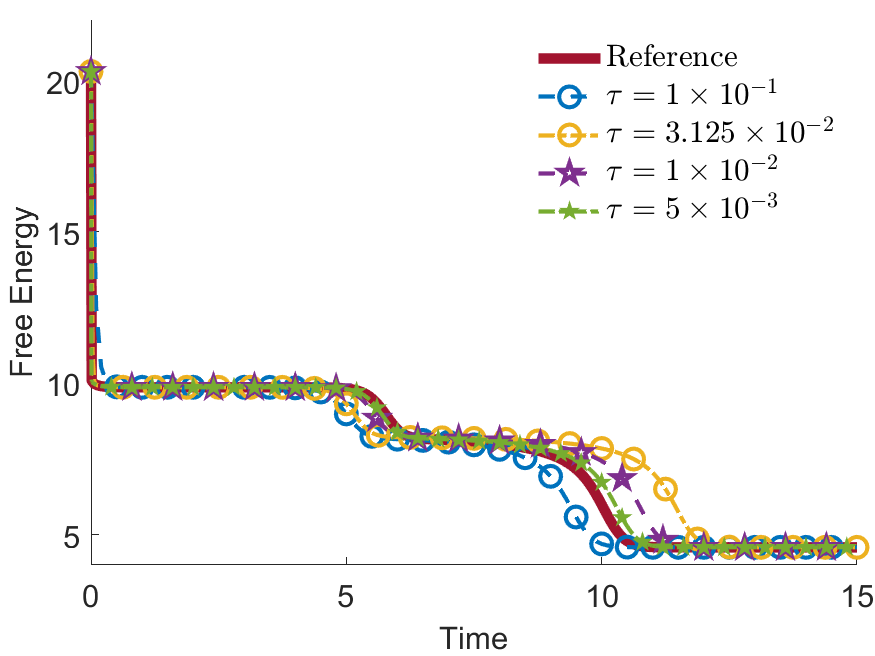}
}
\subfigure[BDF3]{
    \includegraphics[width=0.28\linewidth]{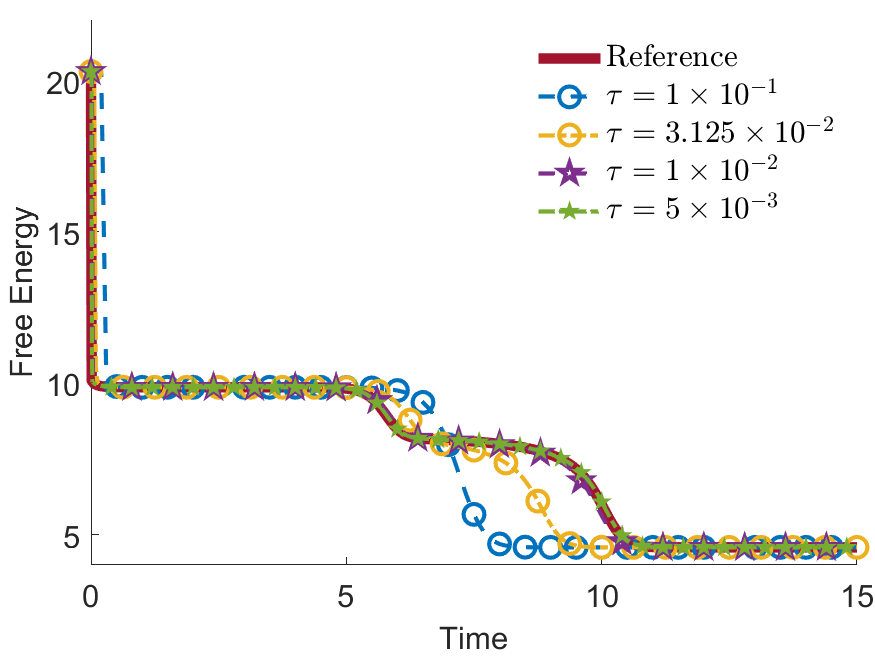} \quad
}
\subfigure[SAV-MDIARK(2,3,3)]{
    \includegraphics[width=0.28\linewidth]{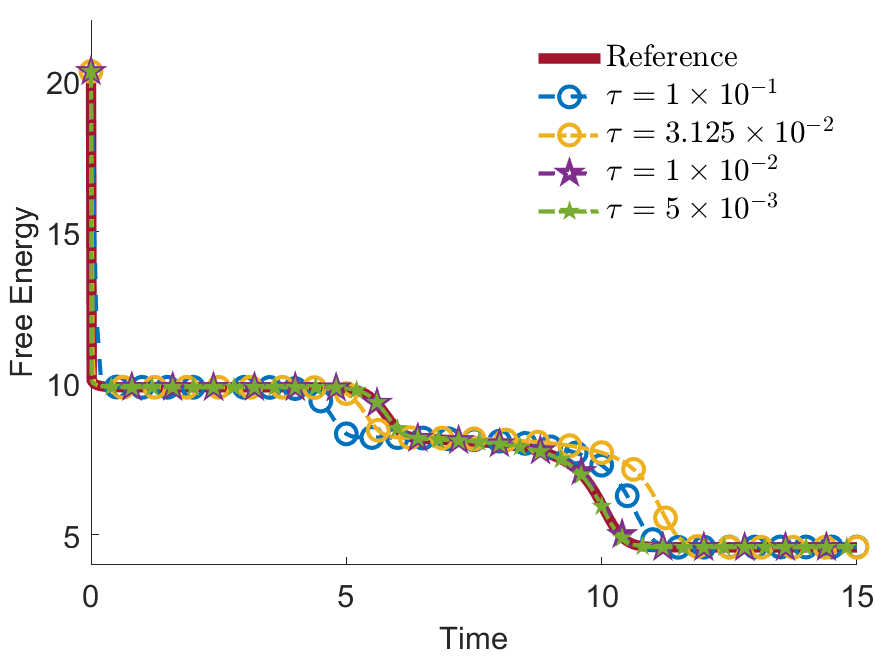} \quad
}
\subfigure[SAV-MDIARK(4,5,3)]{
    \includegraphics[width=0.28\linewidth]{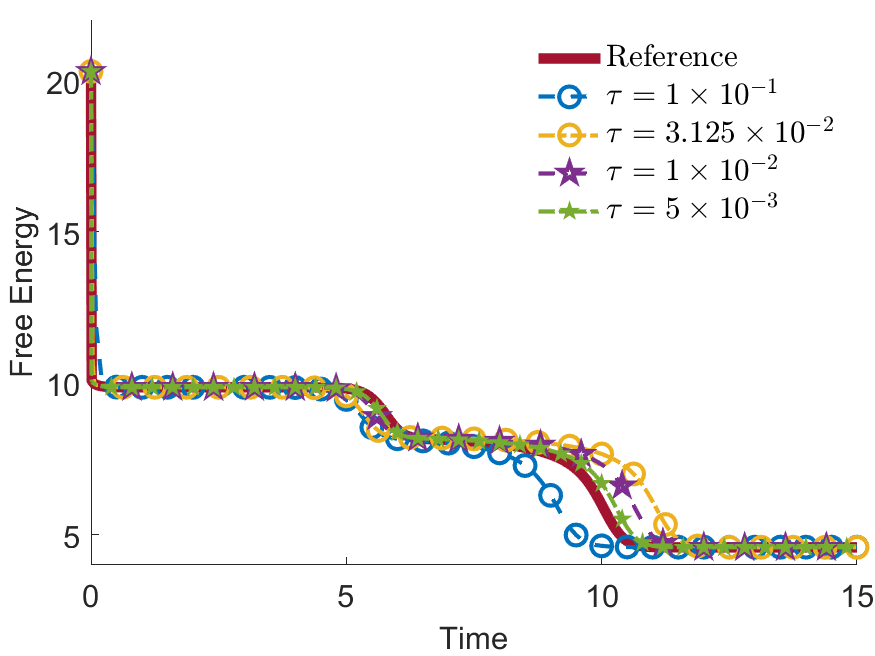}
}
\subfigure[BDF4]{
    \includegraphics[width=0.28\linewidth]{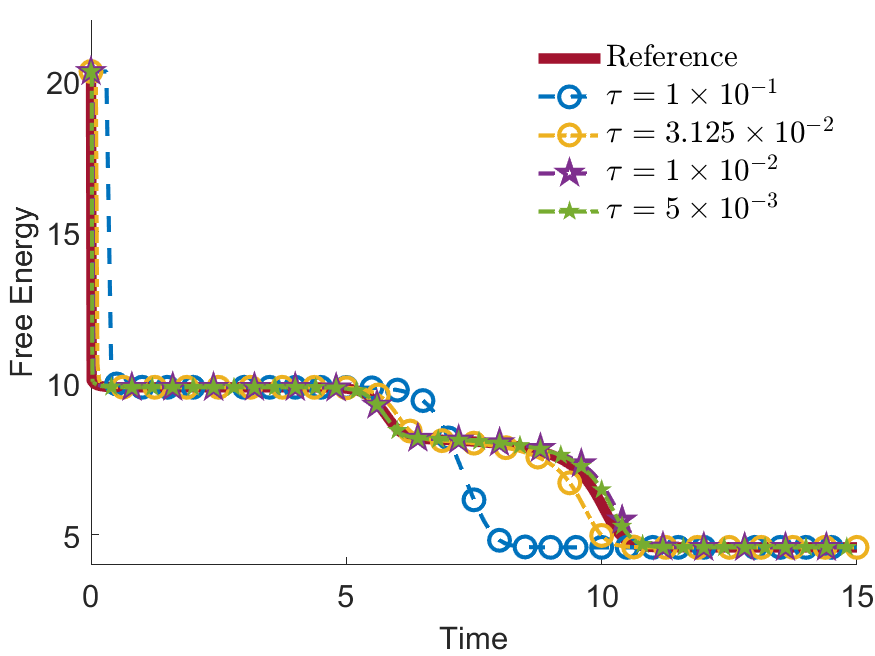} \quad
}
\subfigure[SAV-MAGRK]{
    \includegraphics[width=0.28\linewidth]{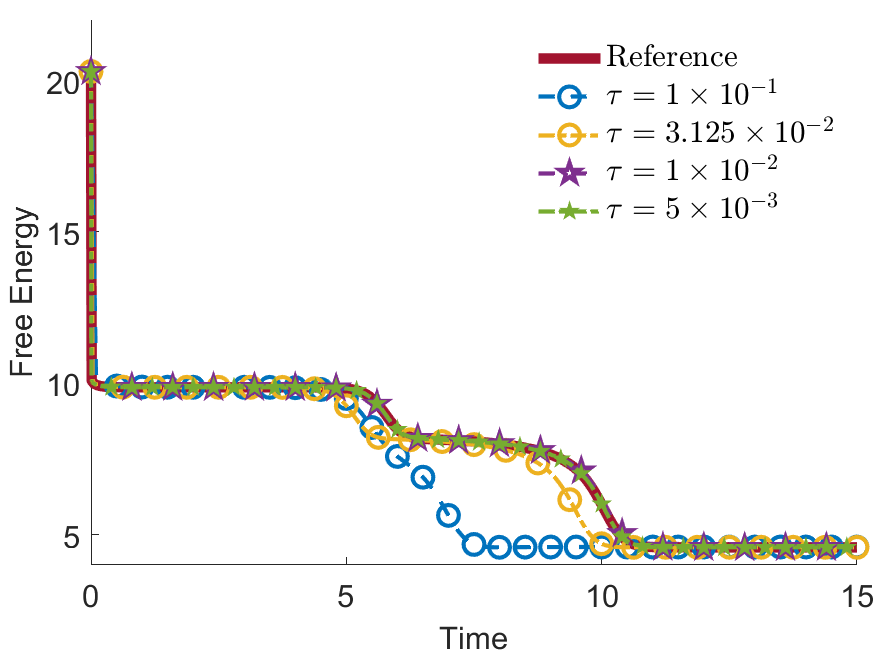} \quad
}
\subfigure[SAV-MDIARK(5,6,4)]{
    \includegraphics[width=0.28\linewidth]{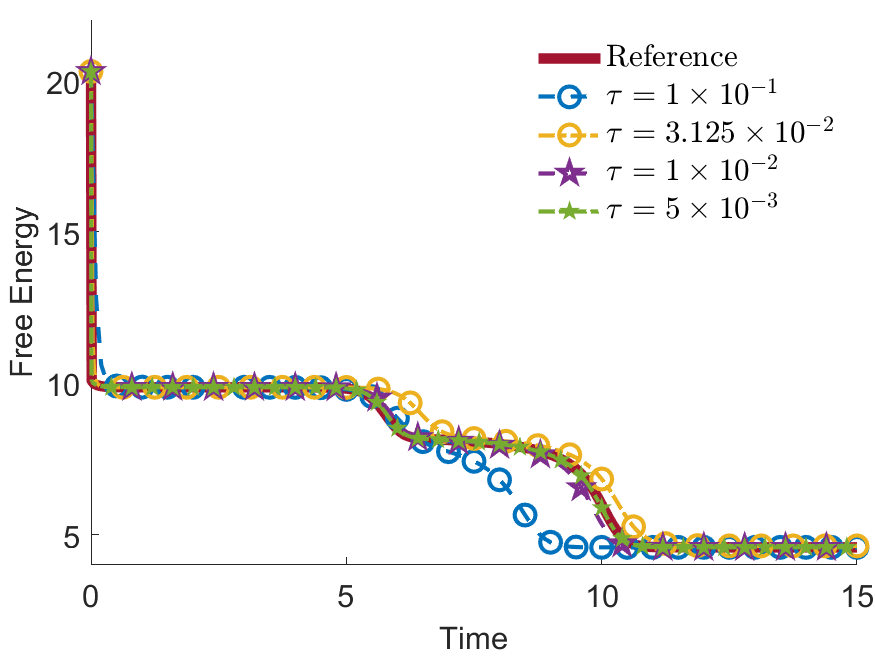}
}
\caption{Time history of the free energy solved by a various of different methods.} \label{fig:mbe_fenergy}
\end{figure}
Figure \ref{fig:mbe_fenergy} summarizes the evolution of free energy from $t = 0$ to $t = 15$ solved by different methods with different time steps. Notice that the energy curve for the fully-implicit backward difference (BDF) methods, which are recognized to have good stability, are also plotted for comparisons. For the third- and fourth-order schemes, the energy curves predicted by the proposed methods are comparable with those predicted by BDF methods. Moreover, among the second-order schemes, the proposed methods provide more accurate energy prediction than the BDF2 method when $\tau = 3.125\times 10^{-2}$. These suggest that our methods are comparable to the fully discrete BDF methods in terms of stability. However, it should be noted that our methods are linearly implicit and only require the solution of a linear system at each step. Table \ref{tab:cpu} lists the CPU times for these methods when conducting the above experiments with the time step of $\tau = 1\times 10^{-2}$. Despite the ARK methods needing to solve more intermediate stages, particularly for higher-order schemes, our proposed methods are more efficient than BDF methods.
\begin{table}[H]
\begin{center}
    \caption{CPU time} \label{tab:cpu}
    \begin{tabular}{c c c c}
        \toprule
        Method & CPU Time(s) & Method & CPU Time(s) \\
        \midrule 
        BDF2 & 112.22 & SAV-MDIARK(2,2,2) & 33.18 \\
        BDF3 & 114.13 & SAV-MDIARK(2,3,3) & 34.89 \\
        BDF4 & 112.28 & SAV-MDIARK(4,5,3) & 75.18 \\ 
        SAV-MAGRK & 63.39 & SAV-MDIARK(5,6,4) & 69.70\\
        \bottomrule
    \end{tabular}
\end{center}
\end{table}

\section{Conclusion} \label{sec6}
Combing the SAV approach and ARK methods, we develop a novel paradigm for constructing linearly implicit and high-order unconditionally energy-stable methods for general gradient flows. The proposed schemes are rigorously proved to be unconditionally energy-stable, uniquely solvable, and convergent. We also reveal that each SAV-RKPC method can be regarded as a SAV-ARK method, and the order of the SAV-RKPC methods are then confirmed theoretically using the order-conditions of ARK methods. Numerical examples demonstrate the efficiency and robustness of the proposed methods.
\section{Acknowledgments}
This work is supported by the National Key Research and Development
Project of China (2018YFC1504205), the National Natural Science Foundation of China (12171245, 11971242).

\begin{appendices}
	\section{Examples of some SAV-ARK methods}
    In this section, we list SAV-ARK methods utilized in the above contexts. We will refer a SAV-DIARK (or SAV-MDIARK) method with $s$-stage implicit method, $r$-stage explicit method and $p$-th order as SAV-DIARK$(s,r,p)$ (SAV-MDIARK$(s,r,p)$).
	\subsection{SAV-DIARK(2,2,2)}
	\begin{equation*}
	\begin{aligned}
		&A = 
		\begin{pmatrix}
			\gamma & 0 \\ 
			1-2\gamma & \gamma
		\end{pmatrix}, \\
		&b = 
		\begin{pmatrix}
			\frac{1}{2} & \frac{1}{2}
		\end{pmatrix}^{\mathrm{T}}, \\ 
		&\widehat{A} = 
		\begin{pmatrix}
			0 & 0 \\ 
			1 & 0
		\end{pmatrix}, \\ 
		&\widehat{b} = b.
    \end{aligned}
	\end{equation*}
	The discriminant of the implicit part of the above method reads
	\begin{equation*}
		\varDelta = (\lambda - \frac{1}{4})
		\begin{pmatrix}
			1 & -1 \\ 
			-1 & 1
		\end{pmatrix}.
	\end{equation*}
	Therefore, the implicit part of the method is algebraic stable iff $\lambda \geq \frac{1}{4}$.

	\subsection{SAV-DIARK(2,3,3)}
	\begin{equation*}
	\begin{aligned}
		& A = 
		\begin{pmatrix}
			 0 & 0 & 0 \\ 
			 0 & \frac{3 + \sqrt{3}}{6} & 0  \\
			 0 &  -\frac{\sqrt{3}}{3} & \frac{3 + \sqrt{3}}{6}
		\end{pmatrix}, \\
		& b = 
		\begin{pmatrix}
			 0 & \frac{1}{2} & \frac{1}{2}
		\end{pmatrix}^\mathrm{T}, \\ 
	 	& \widehat{A} = 
	 	\begin{pmatrix}
	 		 0 & 0 & 0 \\
	 		 \frac{3 + \sqrt{3}}{6} & 0 & 0  \\
	 		 \frac{-3 + \sqrt{3}}{6} &  \frac{3 - \sqrt{3}}{3} & 0
	 	\end{pmatrix}, \\
 		& \widehat{b} = b.
	\end{aligned}
	\end{equation*}
	The eigenvalues of diagonally implicit part are [1.0774, 0, 0, 0].
	\subsection{SAV-DIARK(3,4,3)}
	\begin{equation*}
	\begin{aligned}
		&A = 
		\begin{pmatrix}
			0 & 0 & 0 & 0 \\
			0 & \sigma & 0 & 0 \\
			0 & \frac{1}{2} - \sigma & \sigma & 0  \\
			0 & 2\sigma & 1 - 4\sigma & \sigma \\
		\end{pmatrix}, \\ 
	   &b = 
	   \begin{pmatrix}
	   		0 & \mu & 1-2\mu & \mu
	   \end{pmatrix}^\mathrm{T}, \\ 
   		&\widehat{A} = 
   		\begin{pmatrix}
   			0 & 0 & 0 & 0 \\
   			\sigma & 0 & 0 & 0 \\
   			0 & \frac{1}{2} & 0 & 0  \\
   			0 & \frac{9 \mu \sigma-3 \mu-3 \sigma+1}{3 \mu (2 \sigma-1)} & 1-\sigma- \frac{9 \mu \sigma-3 \mu-3 \sigma+1}{3 \mu (2 \sigma-1)} & 0 \\
   		\end{pmatrix}, \\ 
   		&\widehat{b} = b,
	\end{aligned}
	\end{equation*}
	where 
	\begin{equation*}
		\sigma = \frac{\sqrt{3}}{3} \cos{(\frac{\pi}{18})} + \frac{1}{2}, \quad \mu = \frac{1}{6 (2\sigma - 1)^2}.
	\end{equation*}
	Then, the eigenvalues of the diagonally implicit part are $[1.5530, 0, 0, 0]$.
	\subsection{SAV-DIARK(5,6,4)}
	\begin{equation*}
	\begin{aligned}
		&A = \begin{pmatrix}
			0 & 0 & 0 & 0 & 0 & 0 \\
			0 & \frac{3}{8} & 0 & 0 & 0 & 0 \\
			\frac{3}{8} & 0 & \frac{3}{16} & 0 & 0 & 0 \\
			0 & 0 & 0 & \sigma & 0 & 0 \\
			0 & 0 & 0 & \frac{1}{2} - \sigma & \sigma & 0 \\
			0 & 0 & 0 & 2\sigma & 1 - 4\sigma & \sigma \\
		\end{pmatrix}, \\ 
		&b = 
		\begin{pmatrix}
			0 & 0 & 0 & \mu & 1-2\mu & \mu
		\end{pmatrix}^{\mathrm{T}}, \\ 	
		&\widehat{A} = 
		\begin{pmatrix}
			0 & 0 & 0 & 0 & 0 & 0  \\
			\frac{3}{8} & 0 & 0 & 0 & 0 & 0 \\
			0 & \frac{9}{16} & 0 & 0 & 0 & 0 \\ 
			\frac{25}{162 \mu} & \frac{-104 \sigma \mu^2 +6 \mu^{2}+20 \mu}{108 \mu^{2}-90 \mu+9} & \frac{112 \sigma \mu^2 +36 \mu^{2}-37 \mu}{324 \mu^{2}-270 \mu+27} & 0 & 0 & 0 \\
			 0 & 0 & \frac{1}{2} & 0 & 0 & 0 \\
			 0 & \frac{56 \sigma \mu^2 -2 \mu^{2}-12 \mu}{36 \mu^{2}-30 \mu+3} & \frac{16 \sigma \mu^2 -4 \mu^{2}+3 \mu}{36 \mu^{2}-30 \mu+3} & 0 & 0 & 0 \\
		\end{pmatrix}, \\ 
		&\widehat{b} = b.
	\end{aligned}
	\end{equation*}
    The eigenvalues of the diagonally implicit part are $[1.5530, 0, 0, 0, 0, 0]$.
	\subsection{SAV-GARK(4,5,4)}
	\begin{equation*}
	\begin{aligned}
		&A = 
		\begin{pmatrix}
			0 & 0 & 0 & 0 & 0 \\
			0 & \frac{1}{4} & 0 & 0 & 0 \\
			\frac{1}{4} & 0 & \frac{1}{4} & 0 & 0 \\
			0 & 0 & 0 & \frac{1}{4} & \frac{1}{4}-\frac{\sqrt{3}}{6} \\
			0 & 0 & 0 & \frac{1}{4}+\frac{\sqrt{3}}{6} & \frac{1}{4}
		\end{pmatrix}, \\ 
		& b = 
		\begin{pmatrix}
			0 & 0 & 0 & \frac{1}{2} & \frac{1}{2}
		\end{pmatrix}^{\mathrm{T}}, \\ 
		& \widehat{A} = 
		\begin{pmatrix}
				 0 & 0 & 0 & 0 & 0 \\
			\frac{1}{4} & 0 & 0 & 0 & 0 \\
			0 & \frac{1}{2} & 0 & 0 & 0 \\
			\frac{1}{6} & 0 & \frac{1}{3}-\frac{\sqrt{3}}{6} & 0 & 0 \\
			\frac{1}{6} & 0 & \frac{1}{3}+\frac{\sqrt{3}}{6} & 0 & 0 \\
		\end{pmatrix}, \\ 
		& \widehat{b} = b.
	\end{aligned}
	\end{equation*}
    The implicit part of the above method is based on the Gauss RK method (see \cite{hairer_book}). The eigenvalues of $A$ are $[0, 0, 0, 0, 0]$.
\end{appendices}

\bibliographystyle{abbrv}

\end{document}